\documentclass[10.5pt]{article}

\usepackage[margin=2.8cm]{geometry}
\usepackage{amsmath,amssymb,amsthm}
\usepackage{enumitem,color}
\usepackage{hyperref}
\usepackage{multirow,booktabs}

\usepackage{extarrows}
\usepackage{bm}
\usepackage{graphicx}
\usepackage{graphics}
\usepackage{caption2}
\usepackage{psfrag}
\usepackage{arydshln}

\usepackage{amscd}
\usepackage{amsfonts}
\usepackage{float}
\usepackage{latexsym}
\usepackage{color}
\usepackage{lscape}

\newtheorem{theorem}{Theorem}[section]
\newtheorem{lemma}[theorem]{Lemma}

\newtheorem{remark}{Remark}[section]
\newtheorem{assumption}{Assumption}[section]

\newcommand{\p}{\partial}

\title{Two ADI compact difference methods for variable-exponent diffusion wave equations}

\author{Hao Zhang\thanks{School of Computer Science and Engineering, Sun Yat-sen University, Guangzhou 510006, Guangdong, P. R. China. (Email: zhangh925@mail2.sysu.edu.cn) }
\and
Kexin Li\thanks{Corresponding author. School of Statistics and Mathematics, Yunnan University of Finance and Economics, Kunming 650221, Yunnan, P. R. China. (Email: likx1213@163.com) }
\and
Wenlin Qiu\thanks{Corresponding author. School of Mathematics, Shandong University, Jinan 250100, Shandong, P. R. China. (Email: wlqiu@sdu.edu.cn) }
}

\date{}

\begin{document}

\maketitle

\begin{abstract}
In this work, we study two-dimensional diffusion-wave equations with variable exponent, modeling mechanical diffusive wave propagation in viscoelastic media with spatially varying properties. We first transform the diffusion-wave model into an equivalent form via the convolution method. Two time discretization strategies are then applied to approximate each term in the transformed equation, yielding two fully discrete schemes based on a spatial compact finite difference method. To reduce computational cost, the alternating direction implicit (ADI) technique is employed. We prove that both ADI compact schemes are unconditionally stable and convergent. Under solution regularity, the first scheme achieves $\alpha(0)$-order accuracy in time and fourth-order accuracy in space, while the second scheme attains second-order accuracy in time and fourth-order accuracy in space. Numerical experiments confirm the theoretical error estimates and demonstrate the efficiency of the proposed methods.

\end{abstract}

\section{Introduction}

Fractional partial differential equations (FPDEs), owing to their inherent nonlocal and memory characteristics, have been widely employed to model a variety of physical and engineering processes \cite{wulibeijing1}. They have proved effective in simulating diverse phenomena, such as anomalous diffusion \cite{wulibeijing2}, viscoelasticity \cite{wulibeijing3}, biological systems \cite{wulibeijing4}, and quantum mechanics \cite{wulibeijing5}. Recent studies further suggest that the interactions among different media may evolve in many dynamic processes. In such settings, FPDEs with variable exponents provide a natural modeling framework. Variable-exponent formulations have therefore become increasingly practical and relevant; see, for example, \cite{SunZha,SunCha,wulibeijing6}. For these reasons, FPDEs with variable exponents have attracted growing attention, both in the development of mathematical theory and in numerical approximation \cite{wulibeijing7,GarGiu,Huang,Hong,Jia,ZenZhaKar,Zayernouri,ZhuLiu}.

In this paper, we investigate the following diffusion-wave model with the variable exponent
\begin{equation}\label{VtFDEs}
\begin{array}{c}
(k*\p_t^2 u)(\bm x,t)-  \Delta u (\bm x,t)= f(\bm x,t) ,~~(\bm x,t) \in \Omega\times(0,T],
\end{array}\end{equation}
equipped with initial-boundary conditions
\begin{equation}\label{ibc}
 u(\bm x,0)=u_0(\bm x),~\p_tu(\bm x,0)=\bar u_0(\bm x),~\bm x\in \Omega; \quad u(\bm x,t) = 0,~(\bm x,t) \in \p \Omega\times[0,T]. \end{equation}
Let $\Omega \subset \mathbb{R}^2$ be a simply-connected bounded domain with a piecewise smooth, convex boundary $\partial \Omega$, and let $\bm{x} := (x,y)$ denote the spatial coordinates. The two-dimensional Laplacian is defined by
$\Delta u := \partial_x^2 u + \partial_y^2 u$,
and $f$, $u_0$, and $\bar{u}_0$ denote the source term and initial values, respectively. For simplicity, we focus on a rectangular domain
$\Omega = (l_1, r_1) \times (l_2, r_2)$.
 The variable exponent
fractional derivative of order $1<\alpha(t)<2$ is defined by the symbol $*$ of
convolution \cite{LorHar}
\begin{equation}\nonumber
(k*\p_t^2 u)(\bm x,t):=\int_0^tk(t-s)\p_s^2u(\bm x,s)ds,~~ k(t):=\frac{t^{1-\alpha(t)}}{\Gamma(2-\alpha(t))}.
\end{equation}
Throughout this work, we assume that variable exponent $\alpha(t)$ is three times differentiable with $|\alpha'(t)|+|\alpha''(t)|+|\alpha'''(t)|\le Q$ for some constant $Q>0$. Throughout this work, $Q$ represents a positive constant whose value may differ in various contexts.

It is well known that the theory of fractional differential equations with variable exponents is highly intricate \cite{DieFor}, making closed-form analytical solutions virtually unattainable. Although variable-exponent operators share the nonlocal and weakly singular features of their constant-order counterparts, the absence of convolutional structures creates additional challenges for both mathematical analysis and numerical treatment. In recent years, increasingly accurate and efficient numerical methods have been developed for such problems. For example, Chen et al. \cite{Cchen} proposed a scheme with first-order temporal convergence for variable-exponent fractional advection–diffusion equations using Fourier analysis, while Zeng et al. \cite{ZenZhaKar} employed a spectral collocation method for variable-exponent fractional diffusion equations. Further related contributions include \cite{KiaSocYam,LiWanWan,Liang,Ma,ZheWanJMAA,ZheWanSINUM,ZheWanIMA}.

Although numerous numerical studies address variable-exponent time- and space-fractional problems, the rigorous analysis of diffusion-wave type equations such as \eqref{VtFDEs}-\eqref{ibc} remains scarce. \textit{The primary difficulty lies in the variable-exponent Abel kernel $k(t)$, which cannot be treated analytically by conventional techniques and may fail to be positive definite or monotone.} Some progress has been made for variable-exponent fractional diffusion-wave models involving an additional temporal leading term $\partial_{t}^{2}u$. For instance, Zheng et al. \cite{zheWanCNSNS} employed spectral decomposition to reduce the problem to an ordinary differential equation and designed a first-order temporal scheme, while \cite{duruilian} developed and analyzed a second-order finite difference method. Nevertheless, to the best of the author’s knowledge, very few numerical methods exist for model \eqref{VtFDEs}–\eqref{ibc}, where $k * \p_t^2 u$ serves as the leading term, that attain second-order accuracy in time.

Recently, a convolution-based approach was proposed in \cite{Zhe}, which transforms variable-exponent subdiffusion problems into more tractable formulations, facilitating rigorous analysis. Building on this idea, Qiu and Zheng \cite{Qiu1} reformulated the models \eqref{VtFDEs}-\eqref{ibc} as integro-differential equations and studied their well-posedness and solution regularity. Motivated by these works, we aim to develop appropriate numerical discretizations for models \eqref{VtFDEs}-\eqref{ibc} in two-dimensional domains.

It is well known that, compared with one-dimensional problems, spatial discretization in two dimensions substantially increases computational complexity. To address this, various acceleration techniques have been developed for multidimensional problems \cite{duruilian,Jiachenwang,Meerschaert,Qiu,Tadjeran}. Notably, the ADI method decomposes a multidimensional problem into a set of independent, easily solvable one-dimensional problems, significantly reducing computational cost and enhancing efficiency \cite{Liao,Qiao}. Motivated by these advantages, we employ the ADI approach in this work to construct efficient numerical schemes for solving two-dimensional variable-exponent diffusion-wave models.

The main contributions of this work are summarized as follows.

\textbf{({\romannumeral1})} Building on the convolution approach in \cite{Qiu1,Zhe}, we transform the variable-exponent diffusion-wave model into an equivalent integro-differential formulation. The reconstructed model possesses favorable properties that enable rigorous numerical analysis using suitable time discretization schemes and facilitate efficient computation when combined with spatial discretization methods.

\textbf{({\romannumeral2})}
Compared with \cite{Qiu1}, we impose a weaker condition, $\alpha'(0)=0$, on $\alpha(t)$. This not only relaxes the restrictions on the generalized identity function $g$ (see Lemma \ref{lemma3.1}) but also significantly broadens the admissible range of $\alpha(t)$, allowing the model to capture a wider class of variable-exponent behaviors.

\textbf{({\romannumeral3})}
Using a discrete energy technique, we establish the unconditional stability of the two ADI compact difference schemes and derive their error estimates, with convergence orders $\tau^{\alpha(0)} + h_1^4 + h_2^4$ and $\tau^2 + h_1^4 + h_2^4$, respectively. Numerical experiments confirm that both schemes attain the expected accuracy, while the ADI approach substantially reduces computational cost compared with the standard method.

The remainder of this paper is organized as follows. Section \ref{sec2} presents preliminaries on temporal and spatial discretizations. In Section \ref{sec2.5}, the original model is reconstructed to facilitate the theoretical analysis. In Sections \ref{sec3} and \ref{sec4}, we develop and analyze the formally second-order and accurately second-order ADI compact difference schemes, respectively. Finally, Section \ref{sec5} reports numerical experiments to assess the accuracy and validate the theoretical results.

\section{Preliminaries}\label{sec2}
In this section, we introduce some preliminary notations and lemmas to be
used in the subsequent sections.
\subsection{The temporal preliminaries}
For a positive integer $N$, define a uniform partition on the time interval $[0,T]$
by $t_{n}=n\tau$ for $0\le n\le N$ with $\tau=T/N$. Herein, we set
\begin{equation*}
t_{n-\frac{1}{2}}=\frac{t_{n}+t_{n-1}}{2},\quad 1\le n\le N,
\end{equation*}
and $v^{n}=v(t_{n})$ for $0\le n\le N$. Let ${\mathcal G}:=\{v^{n}\}_{n=0}^{N}$
be a temporal grid space. Then, for $v^{n}\in \mathcal{G}$, we denote the helpful
notations as
\begin{equation*}
v^{n-\frac{1}{2}}=\frac{1}{2}\left(v^{n}+v^{n-1}\right),\quad \delta_{t}v^{n-\frac{1}{2}}=\frac{v^{n}-v^{n-1}}{\tau},\quad 1\le n\le N.
\end{equation*}
Below, we introduce two different numerical approximation methods for handling
integral with a weakly singular kernel, which is defined as
\begin{equation}\label{e2.1}
I^{(\alpha)}\phi(t)=\int_{0}^{t}\frac{(t-s)^{\alpha-1}}{\Gamma(\alpha)}\phi(s)ds,\quad t\in (0,T].
\end{equation}

Firstly, the trapezoidal convolution quadrature rule approximating $I^{(\alpha)}\phi(t_{n})$
is given by
\begin{equation*}
\mathcal{Q}_{n}^{(\alpha)}\phi=\tau^{\alpha} \left( \sum\limits_{p=0}^n \chi_{p}^{(\alpha)} \phi^{n-p} +  \rho_n^{(\alpha)} \phi^0 \right),
\end{equation*}
where the quadrature weights $\chi_{p}^{(\alpha)}$ are obtained from the generating function
\begin{equation}\label{e2.2}
  \left(\frac{2(1-z)}{1+z}\right)^{-\alpha}=\sum_{p=0}^{\infty}\chi_{p}^{(\alpha)}z^{p},
\end{equation}
which follows from \cite[Lemma 5.3]{Lub} that
\begin{equation*}
\chi_{n}^{(\alpha)}=2^{-\alpha}\sum_{s=0}^n\sigma_s^\alpha\varrho_{n-s}^\alpha,\quad n\geq0,
\end{equation*}
with
\begin{equation*}
  \sigma_{s}^{\alpha}=\frac{\Gamma(\alpha+s)}{\Gamma(\alpha)\Gamma(s+1)},\quad\varrho_{s}^{\alpha}=\frac{\Gamma(\alpha+1)}{\Gamma(\alpha-s+1)\Gamma(s+1)},\quad s\geq0.
\end{equation*}
Besides, to ensure the integral formally to the second order, we take the correction
quadrature weights $\rho_{n}^{(\alpha)}$, which can be determined by taking $\phi=1$,
then
\begin{equation*}
\rho_{n}^{(\alpha)}=\frac{n^{\alpha}}{\Gamma(\alpha+1)}-\sum_{p=0}^{n}\chi_{p}^{(\alpha)}.
\end{equation*}
This is the quadrature rule used in the formally second-order ADI compact difference
scheme. Then, there are several important lemmas about this rule.
\begin{lemma}\label{CQerror}\cite{Qiu}
If $\phi$ is real and continuously differentiable when $t\in(0,T]$, and $\partial_{t}^{2}\phi$
is continuous and integrable in $0<t\le T$, then, for $\alpha\in (0,1)$ and $n\ge 1$,
the trapezoidal quadrature error can be bounded by
\begin{equation*}
\begin{split}
\left|I^{(\alpha)}\phi(t_{n})-\mathcal{Q}_{n}^{(\alpha)}\phi\right|\le Q\Biggl[&\tau^{2}t_{n}^{\alpha-1}|\partial_{t}\phi(0)|+\tau^{\alpha+1}\int_{t_{n-1}}^{t_{n}}|\partial_{s}^{2}\phi(s)|ds+\tau^{2}\int_{0}^{t_{n-1}}\left(t_{n}-s\right)^{\alpha-1}|\partial_{s}^{2}\phi(s)|ds\Biggl].
\end{split}
\end{equation*}
\end{lemma}
\begin{lemma}\label{CQpositive}\cite{Lub2}
Let $v^{n}\in \mathcal{G}$ and $\chi_{p}^{(\alpha)}$ be defined in \eqref{e2.2}, then
it holds that
\begin{equation*}
\sum_{n=0}^{N}\left(\sum_{p=0}^{n}\chi_{p}^{(\alpha)}v^{n-p}\right)v^{n}\ge 0.
\end{equation*}
\end{lemma}

After introducing the trapezoidal convolution quadrature rule,
we turn our attention to the second method for handling \eqref{e2.1}.
More precisely, the second method does not directly approximate \eqref{e2.1},
but rather integrates it with respect to $t$ from $t_{n-1}$ to $t_{n}$, and then
multiplies by $1/\tau$, i.e.,
\begin{equation}\label{e2.3}
  \frac{1}{\tau}\int_{t_{n-1}}^{t_{n}}\int_{0}^{t}\beta_{\alpha}(t-s)\phi(s)dsdt,\quad \text{for}~\beta_{\alpha}(t)=\frac{t^{\alpha-1}}{\Gamma(\alpha)}.
\end{equation}
In order to approximate \eqref{e2.3} at the level $t=t_{n}$, we refer to the averaged PI rule suggested by McLean and Mustapha \cite{Mclean},
\begin{equation*}
\mathcal{P}^{n-\frac{1}{2}}_{\alpha}\phi=\frac{1}{\tau}\int_{t_{n-1}}^{t_{n}}I^{(\alpha)}\breve{\phi}(t)dt=\lambda_{n,1}\phi^{1}+\sum_{j=2}^{n}\lambda_{n,j}\phi^{j-\frac{1}{2}},
\end{equation*}
in which we utilize the piecewise constant approximation
\begin{equation*}
\breve{\phi}(t)=
\begin{cases}
  \phi^1,&t_0<t<t_1,\\
  \phi^{n-\frac{1}{2}},&t_{n-1}<t<t_n,\quad n\ge2,
\end{cases}
\end{equation*}
with the weights
\begin{equation*}
  \lambda_{n,j}=\frac{1}{\tau}\int_{t_{n-1}}^{t_n}\int_{t_{j-1}}^{\min\{t,t_j\}}\beta_{\alpha}(t-s)dsdt>0,\quad1\leq j\leq n.
\end{equation*}
This is the PI rule used in the accurate second-order ADI compact difference scheme, and the local truncation errors between $\mathcal{P}_{\alpha}^{n-\frac{1}{2}}\phi$ and
\eqref{e2.3} will be discussed in Section \ref{sec4}. Furthermore, a lemma concerning this rule, formulated in terms of a nonnegative quadratic form, is required and serves as a key ingredient in the subsequent analysis.

\begin{lemma}\label{PIpositive}
For any grid function $v^{n}\in \mathcal{G}$ and $\alpha\in (0,1)$, we have
\begin{equation*}
v^{1}\mathcal{P}_{\alpha}^{\frac{1}{2}}v+\sum_{n=2}^{N}v^{n-\frac{1}{2}}\mathcal{P}_{\alpha}^{n-\frac{1}{2}}v\ge 0.
\end{equation*}
\end{lemma}

\subsection{The spatial preliminaries}
Let spatial step sizes $h_{1}=(r_{1}-l_{1})/{M_{1}}$ and $h_{2}=(r_{2}-l_{2})/{M_{2}}$
for $M_{1}, M_{2}\in \mathbb{Z}^{+}$, we can define a uniform partition of $\Omega$ by
$x_{i}=l_{1}+ih_{1}$ for $0\le i\le M_{1}$ and $y_{j}=l_{2}+jh_{2}$ for $0\le j\le M_{2}$.
We assume that the partition is quasi-uniform, i.e., $Q_{1}h_{1}\leq h_{2} \leq Q_{2}h_{1}$
with $Q_{1}$ and $Q_{2}$ being positive constants independent of $h_{1}$, $h_{2}$ or $\tau$.
Besides, denote $\omega= \{(i,j)\:|\:1 \leq i \leq M_{1}-1, 1 \leq j \leq M_{2}-1\}$ and
$\bar{\omega}= \{(i,j)\:|\:0\leq i \leq M_{1}, 0 \leq j \leq M_{2}\}$ and
$\partial\omega=\bar{\omega}\setminus\omega$. Then, we define spatial grid function
spaces
    \begin{equation*}
        \mathfrak{U}=\{u \mid u=\{u_{i,j}\}, (i,j)\in \bar{\omega}\},\qquad
        \dot{\mathfrak{U}}=\{u\mid u\in\mathfrak{U}~\text{and}~u_{i,j}=0~\text{if}~(i,j)\in\partial\omega\}.
    \end{equation*}
Next, for any grid function $u\in \mathfrak{U}$, we introduce the following discrete operators
\begin{equation*}
\begin{split}
\delta_{x}u_{i+\frac{1}{2},j}=\frac{u_{i+1,j}-u_{i,j}}{h_{1}},\quad\delta_{x}^{2}u_{i,j}=\frac{u_{i+1,j}-2u_{i,j}+u_{i-1,j}}{h_{1}^{2}},\\
\delta_{y}u_{i,j+\frac{1}{2}}=\frac{u_{i,j+1}-u_{i,j}}{h_{2}},\quad\delta_{y}^{2}u_{i,j}=\frac{u_{i,j+1}-2u_{i,j}+u_{i,j-1}}{h_{2}^{2}},
\end{split}
\end{equation*}
\begin{equation*}
\mathcal{A}_{1}u_{i,j}=
\left\{\begin{array}{ccc}
\frac{u_{i-1,j}+10u_{i,j}+u_{i+1,j}}{12},& 1\leq i\leq M_{1}-1,&\quad0\leq j\leq M_{2},\\
u_{i,j},& \quad i=0,M_{1},&\quad0\leq j\leq M_{2},
\end{array}\right.
\end{equation*}
\begin{equation*}
\mathcal{A}_{2}u_{i,j}=
\left\{\begin{array}{ccc}
\frac{u_{i,j-1}+10u_{i,j}+u_{i,j+1}}{12},& 1\leq j\leq M_{2}-1,&\quad0\leq i\leq M_{1},\\
u_{i,j},& \quad j=0,M_{2},&\quad0\leq i\leq M_{1},
\end{array}\right.
\end{equation*}
and
\begin{equation*}
\Delta_{h}u_{i,j}=\delta_{x}^{2}u_{i,j}+\delta_{y}^{2}u_{i,j},\quad\mathcal{A}_{h}u_{i,j}=\mathcal{A}_{1}\mathcal{A}_{2}u_{i,j},\quad\Lambda_{h}u_{i,j}=\mathcal{A}_{2}\delta_{x}^{2}u_{i,j}+\mathcal{A}_{1}\delta_{y}^{2}u_{i,j}.
\end{equation*}

Furthermore, for any grid functions $u$ and $w$ in the grid function space $\dot{\mathfrak{U}}$,
invoke the discrete inner products and corresponding norms
\begin{equation*}
\begin{aligned}
&(u,w)=h_1h_2\sum_{i=1}^{M_{1}-1}\sum_{j=1}^{M_{2}-1}u_{i,j}w_{i,j},\quad \|u\|=\sqrt{(u,u)},\\
&(u,w)_{\mathcal{A}_{h}}=(u,\mathcal{A}_{h}w),\quad (u,w)_{\mathcal{A}_{1}}=(u,\mathcal{A}_{1}w),\quad (u,w)_{\mathcal{A}_{2}}=(u,\mathcal{A}_{2}w),\\
&\|u\|_{\mathcal{A}_{h}}=\sqrt{(u,\mathcal{A}_{h}u)},\quad~ \|u\|_{\mathcal{A}_{1}}=\sqrt{(u,\mathcal{A}_{1}u)},\quad \|u\|_{\mathcal{A}_{2}}=\sqrt{(u,\mathcal{A}_{2}u)},\\
&(\delta_{x}u,\delta_{x}w)=h_{1}h_{2}\sum_{i=0}^{M_{1}-1}\sum_{j=1}^{M_{2}-1}\delta_{x}u_{i+\frac{1}{2},j}\delta_{x}w_{i+\frac{1}{2},j},\quad\|\delta_{x}u\|=\sqrt{(\delta_{x}u,\delta_{x}u)},\\
&(\delta_{y}u,\delta_{y}w)=h_{1}h_{2}\sum_{i=1}^{M_{1}-1}\sum_{j=0}^{M_{2}-1}\delta_{y}u_{i,j+\frac{1}{2}}\delta_{y}w_{i,j+\frac{1}{2}},\quad\|\delta_{y}u\|=\sqrt{(\delta_{y}u,\delta_{y}u)}.
\end{aligned}
\end{equation*}
At last, we give some auxiliary lemmas \cite{duruilian} to facilitate the theoretical analysis.
\begin{lemma}
Let grid functions $u,w\in\dot{\mathfrak{U}}$, then it holds that
\begin{equation*}
(\delta_{x}^{2}u,w)=-(\delta_{x}u,\delta_{x}w),\quad (\delta_{y}^{2}u,w)=-(\delta_{y}u,\delta_{y}w),\quad (\delta_{x}^{2}\delta_{y}^{2}u,w)=(\delta_{x}\delta_{y}u,\delta_{x}\delta_{y}w).
\end{equation*}
\end{lemma}
\begin{lemma}\label{space2}
Assume that grid function $w\in\dot{\mathfrak{U}}$, then we yield
\begin{align*}
&\frac{2}{3}\|w\|^{2}\leq\|w\|_{\mathcal A_{1}}^{2}\leq\|w\|^{2},\quad   \frac{2}{3}\|w\|^{2}\leq\|w\|_{\mathcal A_{2}}^{2}\leq\|w\|^{2},\\
&\frac{1}{3}\|w\|^{2}\leq\|w\|_{\mathcal A_{h}}^{2}\leq\|w\|^{2},\quad   \|\mathcal{A}_{h}w\|\leq\|w\|.
\end{align*}
\end{lemma}
\begin{lemma}\label{space3}
Suppose the function $G(x,y)\in C_{x,y}^{6,6}([l_{1},r_{1}]\times[l_{2},r_{2}])$, we have
\begin{align*}
&\frac{1}{12}\left[\partial_{x}^{2}G(x_{i-1})+10\partial_{x}^{2}G(x_{i})+\partial_{x}^{2}G(x_{i+1})\right]-\frac{1}{h_{1}^{2}}\left[G(x_{i-1})-2G(x_{i})+G(x_{i+1})\right]\\
=&\frac{h_{1}^{4}}{240}\partial_{x}^{6}G(\theta_{i}),\quad \theta_i\in(x_{i-1},x_{i+1}),\quad 1\le i\le M_{1}-1,\\
&\frac{1}{12}\left[\partial_{y}^{2}G(y_{j-1})+10\partial_{y}^{2}G(y_{j})+\partial_{y}^{2}G(y_{j+1})\right]-\frac{1}{h_{2}^{2}}\left[G(y_{j-1})-2G(y_{j})+G(y_{j+1})\right]\\
=&\frac{h_{2}^{4}}{240}\partial_{y}^{6}G(\theta_{j}),\quad \theta_j\in(y_{j-1},y_{j+1}),\quad 1\le j\le M_{2}-1.
\end{align*}
\end{lemma}

\section{Model reconstruction}\label{sec2.5}

For the purpose of numerical analysis, we transform the original model here. We first introduce generalized identity function $g(t)$ \cite{Zhe} as
\begin{align}
(\beta_{\alpha_0-1}*k)(t)
&=\int_0^1\frac{(tz)^{\alpha(0)-\alpha(tz)}}{\Gamma(\alpha(0)-1)\Gamma(2-\alpha(tz))}(1-z)^{\alpha(0)-2}z^{1-\alpha(0)}dz=:g(t),\label{mh7}
\end{align}
with $\alpha_0=\alpha(0)$. It is shown in \cite{Zhe} that $g(0)=1$, $|g|\leq Q$ for $t\in [0,T]$ and
\begin{align}\nonumber
|g'|\leq Q(|\ln t|+1),~~|g^{(m)}|\leq Qt^{-(m-1)},~m=2,3 \;\text{ for }\; t\in (0,T].
\end{align}
Then, following the idea of \cite{Qiu1}, we can transform \eqref{VtFDEs} into the form
\begin{equation}\label{Model2}
\p_t u -\beta_{\alpha_0-1}*  \Delta u=p_u+g(t)\bar u_0,~~p_u:=\beta_{\alpha_0-1}*f-g'*\p_tu,
\end{equation}
equipped with the initial and boundary conditions (\ref{ibc}). From \cite{Qiu1},
we know that the models \eqref{Model2}-\eqref{ibc} and \eqref{VtFDEs}-\eqref{ibc}
are indeed equivalent, which indicates that attention to \eqref{VtFDEs} can be
diverted to \eqref{Model2}. Next, we shall consider the numerical approximations of \eqref{Model2}. Denoting $ (0,1) \ni \bar{\alpha} = \alpha_0 -1$ and $\tilde{u}(t)=u(t)-u_0$, then \eqref{Model2} can be further transformed into
\begin{align}
  & \partial_t\tilde{u} + \int_0^t g'(t-s) \p_s \tilde{u}(s)ds - \int_0^t \beta_{\bar{\alpha}}(t-s) \Delta \tilde{u}(s)ds = \mathcal{F}(t), \label{ModelB} \\
  & \tilde{u}(0) =0, \quad \tilde{u}_t(0) = \bar{u}_0, \label{ModelC}
\end{align}
where $\mathcal{F}(t)= \left( \beta_{\bar{\alpha}}*(\Delta u_0 + f) \right) (t) + g(t)\bar{u}_0$.
For the establishment and analysis of the numerical schemes to this problem, we now impose an additional constraint on $\alpha(t)$ to improve the properties of the generalized identity function $g$.

\vskip 1mm
\begin{lemma}\label{lemma3.1}
  Let $g(t)$ be denoted by \eqref{mh7}. If $\alpha'(0)=0$, then $|g'(t)|\leq Q$ for $t\in [0,T]$ and $|g''(t)|\leq Q(|\ln t|+1)$ for $t\in (0,T]$.
\end{lemma}
\begin{proof}
To bound derivatives of $g$, we denote $\mathcal{H}(t)=(tz)^{\alpha(0)-\alpha(tz)}$. Under the condition $\alpha'(0)=0$, the Taylor expansion gives $\mathcal{H}(t)=e^{-\int_0^{tz}(tz-s)\alpha^{\prime\prime}(s)ds\ln(tz)}$. With the notation $\bar{p}=tz$, we have
\begin{equation*}
\begin{aligned}
|\mathcal{H}^{\prime}(t)|=&\left|\mathcal{H}(t)\partial_{\bar{p}}\left(-\int_{0}^{\bar{p}}(\bar{p}-s)\alpha^{\prime\prime}(s)ds\ln(\bar{p})\right)\frac{d\bar{p}}{dt}\right|\\
=&\left|-\mathcal{H}(t)z\left[\ln(\bar{p})\int_{0}^{\bar{p}}\alpha^{\prime\prime}(s)ds+\frac{1}{\bar{p}}\int_{0}^{\bar{p}}(\bar{p}-s)\alpha^{\prime\prime}(s)ds\right]\right|\leq Q(|\bar{p}\ln\bar{p}|+\bar{p})\leq Q,\\
|\mathcal{H}''(t)|=&\bigg|\mathcal{H}(t) z^2\left[\ln(\bar{p})\int_{0}^{\bar{p}}\alpha^{\prime\prime}(s)ds+\frac{1}{\bar{p}}\int_{0}^{\bar{p}}(\bar{p}-s)\alpha^{\prime\prime}(s)ds\right]^2\\
&-\mathcal{H}(t)z\left[\frac{2z}{\bar{p}}\int_{0}^{\bar{p}}\alpha^{\prime\prime}(s)ds+\ln(\bar{p})\alpha^{\prime\prime}(p)z-\frac{z}{\bar{p}^2}\int_{0}^{\bar{p}}(\bar{p}-s)\alpha^{\prime\prime}(s)ds\right]\bigg|\\
\le & Q(|\bar{p}\ln\bar{p}|+\bar{p})^2+Q(1+\ln(\bar{p}))\le Q(1+\ln(\bar{p})),
\end{aligned}
\end{equation*}
in which, we utilized the fact $|\mathcal{H}(t)|\le Q$. Then, we apply them to bound $g'(t)$ and $g''(t)$ as
\begin{equation*}
\begin{aligned}
|g'|=&\left|\int_0^1\partial_t\left(\frac{(tz)^{\alpha(0)-\alpha(tz)}}{\Gamma(\alpha_{0}-1)\Gamma(2-\alpha(tz))}\right)(1-z)^{\alpha_{0}-2}z^{1-\alpha_{0}}dz\right|
\le  Q\int_0^1 (1-z)^{\alpha_{0}-2}z^{1-\alpha_{0}}dz\le Q,\\
|g''|=&\left|\int_0^1\partial_t^2\left(\frac{(tz)^{\alpha(0)-\alpha(tz)}}{\Gamma(\alpha_{0}-1)\Gamma(2-\alpha(tz))}\right)(1-z)^{\alpha_{0}-2}z^{1-\alpha_{0}}dz\right|\\
\le & Q\int_{0}^{1}\left(|\ln(tz)|+1\right)(1-z)^{\alpha_{0}-2}z^{1-\alpha_{0}}dz\\
\le & Q\left(|\ln t|+1\right)\int_{0}^{1}\frac{|\ln t|+|\ln z|+1}{|\ln t|+1}(1-z)^{\alpha_{0}-2}z^{1-\alpha_{0}}dz\le Q(|\ln t|+1).
\end{aligned}
\end{equation*}
Thus, we complete the proof of this lemma.
\end{proof}

\begin{remark}
    Notably, the above lemma removes the restriction $\alpha''(0)=0$ in \cite[Lemma 3.1]{Qiu1}, allowing for a more general setting. Although $g''(t)$ exhibits a weak singularity in this case, suitable treatment techniques enable rigorous stability and convergence analysis, yielding the desired results. This relaxation of conditions thus extends the applicability of the proposed methods.
\end{remark}

Furthermore, in line with \cite{duruilian,Qiu1}, we introduce the following reasonable assumption on the regularity of $u$, which is essential for establishing the subsequent main results.

\vskip 1mm
\begin{assumption}\label{regularity}
For all $(\bm{x},t)\in\bar{\Omega}\times(0,T]$, \eqref{VtFDEs}-\eqref{ibc} has a unique solution $u$, and there exists a positive constant $Q$ such that
\begin{align*}
&\left|\partial_{z}^{6}u(x,y,t)\right|\leq Q\quad for\;\:z=x,y, \quad  \left|\partial_{x}^{2}\partial_{y}^{2} \partial_{t}u(x,y,t)\right|\leq Q, \quad t\in [0,T], \\
&\left|\partial_{t}^{2}\Delta u(x,y,t)\right|\leq Qt^{\alpha_{0}-2}, \quad \left|\partial_{t}^{3}u(x,y,t)\right|\leq Qt^{\alpha_{0}-3}\quad   for \;\; \alpha_0 \in (1,2), \quad  t\in (0,T].
\end{align*}
\end{assumption}

\section{Formally second-order ADI compact difference scheme}\label{sec3}

In what follows, we proceed to establish and analyze an ADI compact difference scheme with formally second-order accuracy in time and fourth-order accuracy in space.

\subsection{Temporal semi-discretization}
Firstly, we discretize \eqref{ModelB}-\eqref{ModelC} in the temporal direction. Here, we consider \eqref{ModelB} at the time level $t=t_n$,
\begin{align}\label{eq4.1}
  & \partial_t\tilde{u}(t_n) + \int_0^{t_n} g'(t_n-s) \p_{s} \tilde{u}(s) ds - \int_0^{t_n} \beta_{\bar{\alpha}}(t_n-s) \Delta \tilde{u}(s)ds = \mathcal{F}(t_n),
\end{align}
for $1\leq n \leq N$. Following the idea of Crank-Nicolson methods, we have
\begin{equation}\label{eq4.2}
\begin{aligned}
  &\frac{1}{2}\left(\partial_t\tilde{u}(t_n)+\partial_t\tilde{u}(t_n-1)\right)\\
  &+\frac{1}{2}\left(\int_0^{t_n} g'(t_n-s) \p_{s} \tilde{u}(s) ds+\int_0^{t_{n-1}} g'(t_{n-1}-s) \p_{s} \tilde{u}(s) ds\right)\\
  &-\frac{1}{2}\left(\int_0^{t_n} \beta_{\bar{\alpha}}(t_n-s) \Delta \tilde{u}(s)ds+\int_0^{t_{n-1}} \beta_{\bar{\alpha}}(t_{n-1}-s) \Delta \tilde{u}(s)ds\right)\\
  =&\frac{1}{2}\left( \mathcal{F}(t_n)+ \mathcal{F}(t_{n-1})\right),\qquad n\ge 1.
\end{aligned}
\end{equation}

Next, we shall discuss the discretization of the terms from the left-hand side of
\eqref{eq4.2}. For the first term, with the notation $\tilde{u}^n:=\tilde{u}(t_n)$, we utilize the Taylor expansion formula to obtain
\begin{equation}\label{e4.3}
\delta_t \tilde{u}^{n-\frac{1}{2}}=\frac{1}{2}\left(\partial_t\tilde{u}(t_n)+\partial_t\tilde{u}(t_n-1)\right)+(R_1)^{n-\frac{1}{2}},  \quad 1\leq n \leq N,
\end{equation}
where
\begin{equation}\nonumber
\begin{aligned}
(R_{1})^{\frac{1}{2}}= & \frac{1}{2\tau}\int_0^{t_1}(t_1-t)\p_{t}^2\tilde{u}(t)dt-\frac{1}{2\tau}\int_0^{t_1}t\p_{t}^2\tilde{u}(t)dt,\\
(R_1)^{n-\frac{1}{2}} = &\frac{1}{\tau}\int_{t_{n-\frac{1}{2}}}^{t_{n}}\frac{(t_{n}-t)^{2}}{2}\p_{t}^3 \tilde{u}(t)dt+\frac{1}{\tau}\int_{t_{n-1}}^{t_{n-\frac{1}{2}}}\frac{(t_{n-1}-t)^{2}}{2}\p_{t}^3 \tilde{u}(t)dt \\
&+ \frac{1}{2}\int_{t_{n-\frac{1}{2}}}^{t_{n}}(t-t_{n})\p_{t}^3 \tilde{u}(t)dt+\frac{1}{2}\int_{t_{n-1}}^{t_{n-\frac{1}{2}}}(t_{n-1}-t)\p_{t}^3 \tilde{u}(t)dt,\quad n\ge 2.
\end{aligned}
\end{equation}
For the second term of \eqref{eq4.1}, we apply the piecewise linear
interpolating to yield
\begin{align}\nonumber
  & \int_0^{t_n} g'(t_n-s) \p_{s} \tilde{u}(s) ds = \sum\limits_{k=1}^{n} w_{n-k} \delta_{t}\tilde{u}^{k-\frac{1}{2}} + (R_2)^n, \quad 1\leq n \leq N,
\end{align}
where $w_k = \int_{t_k}^{t_{k+1}}g'(t)dt=g(t_{k+1})-g(t_{k})$, and from \cite{Qiu1},
we can get
\begin{equation}\nonumber
\begin{split}
(R_2)^n = \sum\limits_{k=1}^{n}    \int_{t_{k-1}}^{t_k} \Biggl[  &\frac{t_{k-1}-s}{\tau} \int_{s}^{t_k} g''(t_n-\theta)(t_k-\theta)d\theta
+ \frac{s-t_k}{\tau} \int_{s}^{t_{k-1}} g''(t_n-\theta)(t_{k-1}-\theta)d\theta   \Biggl] \p_{s}^2 \tilde{u}(s)ds.
\end{split}
\end{equation}
Denote
\begin{equation*}
\tilde{w}_{0}=\frac{w_{0}}{2},\qquad \tilde{w}_{k}=\frac{1}{2}\left(w_{k}+w_{k-1}\right),\quad 1\le k\le n-1.
\end{equation*}
Thus, for the second term of \eqref{eq4.2}, we have
\begin{equation}\label{eq4.4}
\sum\limits_{k=1}^{n} \tilde{w}_{n-k} \delta_{t}\tilde{u}^{k-\frac{1}{2}}
=\frac{1}{2}\left(\int_0^{t_n} g'(t_n-s) \p_{s} \tilde{u}(s) ds+\int_0^{t_{n-1}} g'(t_{n-1}-s) \p_{s} \tilde{u}(s) ds\right)-(R_2)^{n-\frac{1}{2}}.
\end{equation}
Finally, for the third term of \eqref{eq4.2}, we employ the trapezoidal convolution
quadrature rule to get
\begin{equation}\label{eq4.5}
\frac{1}{2}\left(\int_0^{t_n} \beta_{\bar{\alpha}}(t_n-s) \Delta \tilde{u}(s)ds+\int_0^{t_{n-1}} \beta_{\bar{\alpha}}(t_{n-1}-s) \Delta \tilde{u}(s)ds\right)
=\tau^{\bar{\alpha}} \left( \sum\limits_{p=0}^n \chi_{p}^{(\bar{\alpha})} \Delta \tilde{u}^{n-p-\frac{1}{2}} +  \hat{\rho}_n^{(\bar{\alpha})} \Delta \tilde{u}^0 \right) + (R_3)^{n-\frac{1}{2}},
\end{equation}
with the notations $\tilde{u}^{-1}=0$ and
$\hat{\rho}_{n}^{(\bar{\alpha})}=\frac{1}{2}\left(\rho_{n}^{(\bar{\alpha})}+\rho_{n-1}^{(\bar{\alpha})}\right)$
for $1\leq n \leq N$, and
\begin{equation*}
  (R_{3})^{n-\frac{1}{2}}=\frac{1}{2}\left[\left(I^{(\bar{\alpha})}\Delta\tilde{u}(t_{n})-\mathcal{Q}_{n}^{(\bar{\alpha})}\Delta\tilde{u}\right)+\left(I^{(\bar{\alpha})}\Delta\tilde{u}(t_{n-1})-\mathcal{Q}_{n-1}^{(\bar{\alpha})}\Delta\tilde{u}\right)\right].
\end{equation*}
Next, taking \eqref{e4.3}, \eqref{eq4.4} and \eqref{eq4.5} into \eqref{eq4.2},
we have
\begin{equation}\label{e4.6}
\begin{aligned}
\delta_t \tilde{u}^{n-\frac{1}{2}} + \sum\limits_{k=1}^{n} \tilde{w}_{n-k} \delta_t \tilde{u}^{k-\frac{1}{2}} - \tau^{\bar{\alpha}} \left( \sum\limits_{p=0}^n \chi_{p}^{(\bar{\alpha})} \Delta \tilde{u}^{n-p-\frac{1}{2}} +  \hat{\rho}_n^{(\bar{\alpha})} \Delta \tilde{u}^0 \right)
= \frac{1}{2}\left(\mathcal{F}(t_n)+\mathcal{F}(t_{n-1})\right)+R^{n-\frac{1}{2}},
\end{aligned}
\end{equation}
with $1 \leq n \leq N$, where $R^n = (R_1)^n-(R_2)^n+(R_3)^n$.

Then, based on this time discretization, we establish a fully discrete scheme.

\subsection{Fully discrete scheme}
Let $\tilde{u}_{i,j}^{n}:=\tilde{u}(x_{i},y_{j},t_{n})$,
$\mathcal{F}_{i,j}^{n}:=\mathcal{F}(x_{i},y_{j},t_{n})$ for $(i,j)\in \omega$ and
$0\le n\le N$. Based on Lemma \ref{space3} and Assumption \ref{regularity}, we apply the compact operator $\mathcal{A}_{h}$ on both sides of \eqref{e4.6} evaluated at $(x_{i},y_{j})$ to reformulate this formula as follows
\begin{equation}\label{e4.7}
\begin{aligned}
&\delta_t\mathcal{A}_{h}\tilde{u}_{i,j}^{n-\frac{1}{2}} + \sum\limits_{k=1}^{n} \tilde{w}_{n-k} \delta_t\mathcal{A}_{h}\tilde{u}_{i,j}^{k-\frac{1}{2}} - \tau^{\bar{\alpha}} \left( \sum\limits_{p=0}^n \chi_{p}^{(\bar{\alpha})} \Lambda_{h} \tilde{u}_{i,j}^{n-p-\frac{1}{2}} +  \hat{\rho}_n^{(\bar{\alpha})} \Lambda_{h} \tilde{u}_{i,j}^0 \right) \\
= &\mathcal{A}_{h}\mathcal{F}_{i,j}^{n-\frac{1}{2}}+\mathcal{A}_{h}R^{n-\frac{1}{2}}_{i,j}+\mathcal{O}(h_{1}^{4}+h_{2}^{4}),
\end{aligned}
\end{equation}
for $(i,j)\in \omega$ and $1\le n\le N$. At this point, with the help of the Taylor series
expansion with integral remainder, it is not difficult to find
\begin{equation*}
0=\delta_{x}^{2}\delta_{y}^{2}\delta_{t}\tilde{u}_{i,j}^{n-\frac{1}{2}}-(\hat{R}_{4})_{i,j}^{n-\frac{1}{2}}, \quad (i,j)\in \omega,\quad 1\le n\le N,
\end{equation*}
where
\begin{equation*}
  (\hat{R}_{4})_{i,j}^{n-\frac{1}{2}}=\frac{1}{2}\int_{0}^{1}\delta_{x}^{2}\delta_{y}^{2}\biggl[\partial_{s}\tilde{u}(x_{i},y_{j},t_{n}-\frac{s\tau}{2})+\partial_{s}\tilde{u}(x_{i},y_{j},t_{n}+\frac{s\tau}{2})\biggr]ds.
\end{equation*}
Thus, by Assumption \ref{regularity}, setting $c_{0}=1+\frac{w_{0}}{2}$
and putting the small term
\begin{equation*}
\frac{\left(\tau^{\bar{\alpha}+1}\chi_{0}^{(\bar{\alpha})}\right)^{2}}{4c_{0}}\delta_{x}^{2}\delta_{y}^{2}\delta_{t}\tilde{u}_{i,j}^{n-\frac{1}{2}}:=(R_{4})_{i,j}^{n-\frac{1}{2}}
\end{equation*}
into the both sides of \eqref{e4.7}, we can rewrite this formula as
\begin{equation}\label{eq4.8}
\begin{aligned}
&c_{0}\delta_t\mathcal{A}_{h}\tilde{u}_{i,j}^{n-\frac{1}{2}}+ \sum\limits_{k=1}^{n-1} \tilde{w}_{n-k} \delta_t\mathcal{A}_{h}\tilde{u}_{i,j}^{k-\frac{1}{2}} - \tau^{\bar{\alpha}} \left( \sum\limits_{p=0}^n \chi_{p}^{(\bar{\alpha})} \Lambda_{h} \tilde{u}_{i,j}^{n-p-\frac{1}{2}} +  \hat{\rho}_n^{(\bar{\alpha})} \Lambda_{h} \tilde{u}_{i,j}^0 \right) \\
&+\frac{\left(\tau^{\bar{\alpha}+1}\chi_{0}^{(\bar{\alpha})}\right)^{2}}{4c_{0}}\delta_{x}^{2}\delta_{y}^{2}\delta_{t}\tilde{u}_{i,j}^{n-\frac{1}{2}}=\mathcal{A}_{h}\mathcal{F}_{i,j}^{n-\frac{1}{2}}+\tilde{R}_{i,j}^{n-\frac{1}{2}},\quad (i,j)\in\omega,\quad 1\le n\le N,
\end{aligned}
\end{equation}
with $\tilde{R}_{i,j}^{n-\frac{1}{2}}=\mathcal{A}_{h}R_{i,j}^{n-\frac{1}{2}}+\mathcal{O}(\tau^{\alpha_{0}}+h_{1}^{4}+h_{2}^{4})$.
Next, omitting the local truncation error term $\tilde{R}_{i,j}^{n-\frac{1}{2}}$ in
\eqref{eq4.8}, and replacing $\tilde{u}_{i,j}^{n}$ with its numerical approximation
$\widetilde{U}_{i,j}^{n}$, we obtain the formally second-order ADI compact difference
scheme as follows
\begin{equation}\label{e4.9}
\begin{aligned}
&c_{0}\delta_t\mathcal{A}_{h}\widetilde{U}_{i,j}^{n-\frac{1}{2}}+ \sum\limits_{k=1}^{n-1} \tilde{w}_{n-k} \delta_t\mathcal{A}_{h}\widetilde{U}_{i,j}^{k-\frac{1}{2}} - \tau^{\bar{\alpha}} \left( \sum\limits_{p=0}^n \chi_{p}^{(\bar{\alpha})} \Lambda_{h} \widetilde{U}_{i,j}^{n-p-\frac{1}{2}} +  \hat{\rho}_n^{(\bar{\alpha})} \Lambda_{h} \widetilde{U}_{i,j}^0 \right) \\
&+\frac{\left(\tau^{\bar{\alpha}+1}\chi_{0}^{(\bar{\alpha})}\right)^{2}}{4c_{0}}\delta_{x}^{2}\delta_{y}^{2}\delta_{t}\widetilde{U}_{i,j}^{n-\frac{1}{2}}=\mathcal{A}_{h}\mathcal{F}_{i,j}^{n-\frac{1}{2}},\quad (i,j)\in\omega,\quad 1\le n\le N,
\end{aligned}
\end{equation}
which subjects to the conditions as follows
\begin{equation}\label{e4.10}
\widetilde{U}_{i,j}^{0}=0,\quad (i,j)\in \omega,\qquad  \widetilde{U}_{i,j}^{n}=0,\quad (i,j)\in \partial\omega,~0\le n\le N.
\end{equation}

In fact, after getting $\widetilde{U}_{i,j}^{n}$, we can further obtain the numerical solution of \eqref{Model2} through $U_{i,j}^{n}=\widetilde{U}_{i,j}^{n}+u_{0}(x_i,y_j)$. For computation, we rewrite \eqref{e4.9} into a more familiar ADI form. First, denote
\begin{equation*}
\eta_{0}=\frac{\tau^{\bar{\alpha}+1}\chi_{0}^{(\bar{\alpha})}}{2c_{0}},\quad \widetilde{\mathcal{E}}_{i,j}^{n}=\widetilde{U}_{i,j}^{n}-\widetilde{U}_{i,j}^{n-1},\quad 1\le n\le N.
\end{equation*}
Under the circumstances, we multiply \eqref{e4.9} by $\tau$ to reformulate the
equation as
\begin{equation}\label{e4.11}
c_{0}\mathcal{A}_{h}\widetilde{\mathcal{E}}_{i,j}^{n}-\frac{\tau^{\bar{\alpha}+1}\chi_{0}^{(\bar{\alpha})}}{2}\Lambda_{h}\widetilde{\mathcal{E}}_{i,j}^{n}+\frac{\left(\tau^{\bar{\alpha}+1}\chi_{0}^{(\bar{\alpha})}\right)^{2}}{4c_{0}}\delta_{x}^{2}\delta_{y}^{2}\widetilde{\mathcal{E}}_{i,j}^{n}=\widetilde{F}_{i,j}^{n},
\qquad (i,j)\in\omega,\quad 1\le n\le N,
\end{equation}
where
\begin{equation*}
\widetilde{F}_{i,j}^{n}=c_{0}\eta_{0}\Lambda_{h}\widetilde{U}_{i,j}^{n-1}+\tau^{\bar{\alpha}+1} \left( \sum\limits_{p=0}^{n-1}\hat{\chi}_{n-p}^{(\bar{\alpha})} \Lambda_{h} \widetilde{U}_{i,j}^{p} +  \hat{\rho}_n^{(\bar{\alpha})} \Lambda_{h} \widetilde{U}_{i,j}^0 \right)
-\sum\limits_{k=1}^{n-1} \tilde{w}_{n-k}\mathcal{A}_{h}\widetilde{\mathcal{E}}_{i,j}^{k}+\tau \mathcal{A}_{h}\mathcal{F}_{i,j}^{n-\frac{1}{2}},
\end{equation*}
with the notation $\hat{\chi}_{p}^{(\bar{\alpha})}=\frac{1}{2}\left(\chi_{p}^{(\bar{\alpha})}+\chi_{p-1}^{(\bar{\alpha})}\right)$
for $1\le p\le N$. It is easy to split \eqref{e4.11} for $(i,j)\in\omega$ as follows
\begin{equation}\nonumber
c_{0}\left(\mathcal{A}_{1}-\eta_{0}\delta_{x}^{2}\right)\left(\mathcal{A}_{2}-\eta_{0}\delta_{y}^{2}\right)\widetilde{\mathcal{E}}_{i,j}^{n}=\widetilde{F}_{i,j}^{n},\quad (i,j)\in\omega,\quad 1\le n\le N.
\end{equation}
We determine $\widetilde{U}_{i,j}^{n}$ by solving two sets of independent one-dimensional problems. Here, we introduce intermediate variables
\begin{equation*}
\widetilde{\mathcal{E}}_{i,j}^{n,*}=\left(\mathcal{A}_{2}-\eta_{0}\delta_{y}^{2}\right)\widetilde{\mathcal{E}}_{i,j}^{n},\quad (i,j)\in\omega,\quad 1\le n\le N,
\end{equation*}
from which the proposed formally second-order ADI compact difference scheme
\eqref{e4.9}-\eqref{e4.10} can be computed by the following two steps.\\
$\textbf{Step 1:}$ For any fixed $j\in \{1,2,\dots,M_{2}-1\}$, solve the
following one-dimensional linear system in the $x$-direction for $\widetilde{\mathcal{E}}_{i,j}^{n,*}$:
\begin{equation*}
\left\{\begin{array}{l}
c_{0}\left(\mathcal{A}_{1}-\eta_{0}\delta_{x}^{2}\right)\widetilde{\mathcal{E}}_{i,j}^{n,*}=\widetilde{F}_{i,j}^{n},   \quad 1\le i\le M_{1}-1,\\
\widetilde{\mathcal{E}}_{0,j}^{n,*}=\widetilde{\mathcal{E}}_{M_{1},j}^{n,*}=0.
\end{array}\right.
\end{equation*}
$\textbf{Step 2:}$ For each fixed $i\in \{1,2,\dots,M_{1}-1\}$, solve the
following one-dimensional linear system in the $y$-direction for $\widetilde{\mathcal{E}}_{i,j}^{n}$:
\begin{equation*}
\left\{\begin{array}{l}
\left(\mathcal{A}_{2}-\eta_{0}\delta_{y}^{2}\right)\widetilde{\mathcal{E}}_{i,j}^{n}=\widetilde{\mathcal{E}}_{i,j}^{n,*},   \quad 1\le j\le M_{2}-1,\\
\widetilde{\mathcal{E}}_{i,0}^{n}=\widetilde{\mathcal{E}}_{i,M_{2}}^{n}=0,
\end{array}\right.
\end{equation*}
and then find numerical solutions by $\widetilde{U}_{i,j}^{n}=\widetilde{\mathcal{E}}_{i,j}^{n}+\widetilde{U}_{i,j}^{n-1}$ for $(i,j)\in \omega$ and $1\le n\le N$.

\subsection{Analysis of fully discrete scheme}
Below we shall consider the stability and convergence of fully discrete
ADI compact difference scheme \eqref{e4.9}-\eqref{e4.10} by energy technique.
\begin{theorem}
Suppose that $\{\widetilde{U}_{i,j}^{n}\mid(i,j)\in\omega,~1\le n\le N\}$ is the
solution of the formally second-order ADI compact difference scheme \eqref{e4.9}-\eqref{e4.10}, then it holds that
\begin{equation*}
\left\|\widetilde{U}^N\right\| \leq Q \sum\limits_{n=1}^{N} \tau \|\mathcal{F}^{n-\frac{1}{2}}\| \leq Q \left( \|\Delta u_0\| + \| \bar{u}_0\| + \max\limits_{0\leq t \leq t_N}\|f(t)\| \right).
\end{equation*}
\end{theorem}
\begin{remark}
Since $U_{i,j}^{n}=\widetilde{U}_{i,j}^{n}+u_{0}(x_i,y_j)$, we immediately obtain
that the numerical solution to \eqref{Model2} satisfies
\begin{equation*}
\left\|U^N\right\| \leq Q\left( \| u_0\|+ \|\Delta u_0\| + \| \bar{u}_0\| + \max\limits_{0\leq t \leq t_N}\|f(t)\| \right).
\end{equation*}
\end{remark}
\begin{proof} Taking the inner product of \eqref{e4.9} with $2\tau\widetilde{U}^{n-\frac{1}{2}}$,
summing up for $n$ from $1$ to $m$ for $1\le m\le N$,
and noting the fact that $\widetilde{U}^{0}=0$, we have
\begin{equation}\label{e4.12}
\begin{aligned}
& 2\tau c_{0}\sum\limits_{n=1}^{m} \left(\mathcal{A}_{h}\delta_{t}\widetilde{U}^{n-\frac{1}{2}}, \widetilde{U}^{n-\frac{1}{2}} \right) +
2\tau \sum\limits_{n=1}^{m} \sum\limits_{k=1}^{n-1} \tilde{w}_{n-k}  \left(\mathcal{A}_{h}\delta_t \widetilde{U}^{k-\frac{1}{2}}, \widetilde{U}^{n-\frac{1}{2}} \right) \\
& - 2\tau^{\bar{\alpha}+1} \sum\limits_{n=0}^{m} \sum\limits_{p=0}^n  \chi_{p}^{(\bar{\alpha})}  \left( \Lambda_{h}\widetilde{U}^{n-p-\frac{1}{2}}, \widetilde{U}^{n-\frac{1}{2}} \right)
+\frac{2\tau\left(\tau^{\bar{\alpha}+1}\chi_{0}^{(\bar{\alpha})}\right)^{2}}{4c_{0}}\sum\limits_{n=1}^{m}\left(\delta_{x}^{2}\delta_{y}^{2}\delta_{t}\widetilde{U}^{n-\frac{1}{2}},\widetilde{U}^{n-\frac{1}{2}}\right)\\
=& 2\tau \sum\limits_{n=1}^{m} \left( \mathcal{A}_{h}\mathcal{F}^{n-\frac{1}{2}},  \widetilde{U}^{n-\frac{1}{2}} \right).
\end{aligned}
\end{equation}
First, for the first term on the left-hand side of \eqref{e4.12}, it is obvious that
\begin{equation*}
\left(\mathcal{A}_{h}\delta_{t}\widetilde{U}^{n-\frac{1}{2}}, \widetilde{U}^{n-\frac{1}{2}} \right)=\frac{1}{2\tau}\left((\mathcal{A}_{h}\widetilde{U}^{n},\widetilde{U}^{n})-(\mathcal{A}_{h}\widetilde{U}^{n-1},\widetilde{U}^{n-1})\right),
\end{equation*}
from which it follows that
\begin{equation}\label{e4.13}
2\tau c_{0}\sum\limits_{n=1}^{m} \left(\mathcal{A}_{h}\delta_{t}\widetilde{U}^{n-\frac{1}{2}}, \widetilde{U}^{n-\frac{1}{2}} \right)=c_{0}\sum\limits_{n=1}^{m} \left(\left\|\widetilde{U}^{n}\right\|_{\mathcal{A}_{h}}^{2}-\left\|\widetilde{U}^{n-1}\right\|_{\mathcal{A}_{h}}^{2}\right)=c_{0}\left\|\widetilde{U}^{m}\right\|_{\mathcal{A}_{h}}^{2}.
\end{equation}
Then, for the second term, taking the fact
\begin{equation*}
\begin{split}
\sum\limits_{k=1}^{n-1} \tilde{w}_{n-k} \delta_t \widetilde{U}^{k-\frac{1}{2}} &= \frac{1}{\tau} \left[ \tilde{w}_1 \widetilde{U}^{n-1} + \sum\limits_{k=1}^{n-2}(\tilde{w}_{n-k}-\tilde{w}_{n-k-1})\widetilde{U}^k -\tilde{w}_{n-1}\widetilde{U}^0  \right]\\
&= \frac{1}{\tau} \left[ \tilde{w}_1 \widetilde{U}^{n-1} + \sum\limits_{k=2}^{n-1}(\tilde{w}_{k}-\tilde{w}_{k-1})\widetilde{U}^{n-k} \right]
\end{split}
\end{equation*}
into consideration, with the help of the Cauchy-Schwarz inequality, we can obtain
\begin{equation}\label{e4.14}
\begin{aligned}
2\tau \sum\limits_{n=1}^{m} \sum\limits_{k=1}^{n-1} \tilde{w}_{n-k}  \left(\mathcal{A}_{h}\delta_t \widetilde{U}^{k-\frac{1}{2}}, \widetilde{U}^{n-\frac{1}{2}} \right)
\le & \sum\limits_{n=2}^{m}|\tilde{w}_{1}|\left\|\widetilde{U}^{n-1}\right\|_{\mathcal{A}_{h}}\left\|\widetilde{U}^{n}\right\|_{\mathcal{A}_{h}}+\sum\limits_{n=3}^{m}\sum\limits_{k=2}^{n-1}|\tilde{w}_{k}-\tilde{w}_{k-1}|\left\|\widetilde{U}^{n-k}\right\|_{\mathcal{A}_{h}}\left\|\widetilde{U}^{n}\right\|_{\mathcal{A}_{h}}\\
&+\sum\limits_{n=2}^{m}|\tilde{w}_{1}|\left\|\widetilde{U}^{n-1}\right\|_{\mathcal{A}_{h}}^{2}+\sum\limits_{n=3}^{m}\sum\limits_{k=2}^{n-1}|\tilde{w}_{k}-\tilde{w}_{k-1}|\left\|\widetilde{U}^{n-k}\right\|_{\mathcal{A}_{h}}\left\|\widetilde{U}^{n-1}\right\|_{\mathcal{A}_{h}}.
\end{aligned}
\end{equation}
Next, for the third term on the left-hand side of \eqref{e4.12}, utilizing Lemma
\ref{CQpositive} yields
\begin{equation}\label{e4.15}
\begin{aligned}
  &- \sum\limits_{n=0}^{m} \sum\limits_{p=0}^n  \chi_{p}^{(\bar{\alpha})}  \left( \Lambda_{h}\widetilde{U}^{n-p-\frac{1}{2}}, \widetilde{U}^{n-\frac{1}{2}} \right)\\
=& - \sum\limits_{n=0}^{m} \sum\limits_{p=0}^n  \chi_{p}^{(\bar{\alpha})}  \left( \mathcal{A}_{2}\delta_{x}^{2}\widetilde{U}^{n-p-\frac{1}{2}}, \widetilde{U}^{n-\frac{1}{2}} \right)- \sum\limits_{n=0}^{m} \sum\limits_{p=0}^n  \chi_{p}^{(\bar{\alpha})}  \left( \mathcal{A}_{1}\delta_{y}^{2}\widetilde{U}^{n-p-\frac{1}{2}}, \widetilde{U}^{n-\frac{1}{2}} \right)\\
=& \sum\limits_{n=0}^{m} \sum\limits_{p=0}^n  \chi_{p}^{(\bar{\alpha})}  \left( \delta_{x}\widetilde{U}^{n-p-\frac{1}{2}}, \delta_{x}\widetilde{U}^{n-\frac{1}{2}} \right)_{\mathcal{A}_{2}}+\sum\limits_{n=0}^{m} \sum\limits_{p=0}^n  \chi_{p}^{(\bar{\alpha})}  \left( \delta_{y}\widetilde{U}^{n-p-\frac{1}{2}}, \delta_{y}\widetilde{U}^{n-\frac{1}{2}} \right)_{\mathcal{A}_{1}}
\ge 0.
\end{aligned}
\end{equation}
Finally, for the last term, it holds that
\begin{equation}\label{e4.16}
\begin{aligned}
 \frac{2\tau\left(\tau^{\bar{\alpha}+1}\chi_{0}^{(\bar{\alpha})}\right)^{2}}{4c_{0}}\sum\limits_{n=1}^{m}\left(\delta_{x}^{2}\delta_{y}^{2}\delta_{t}\widetilde{U}^{n-\frac{1}{2}},\widetilde{U}^{n-\frac{1}{2}}\right)
=&\frac{\left(\tau^{\bar{\alpha}+1}\chi_{0}^{(\bar{\alpha})}\right)^{2}}{4c_{0}}\sum\limits_{n=1}^{m}\left(\delta_{x}^{2}\delta_{y}^{2}(\widetilde{U}^{n}-\widetilde{U}^{n-1}),\widetilde{U}^{n}+\widetilde{U}^{n-1}\right)\\
=&\frac{\left(\tau^{\bar{\alpha}+1}\chi_{0}^{(\bar{\alpha})}\right)^{2}}{4c_{0}}\sum\limits_{n=1}^{m}\left(\delta_{x}\delta_{y}(\widetilde{U}^{n}-\widetilde{U}^{n-1}),\delta_{x}\delta_{y}(\widetilde{U}^{n}+\widetilde{U}^{n-1})\right)\\
=&c_{0}\eta_{0}^{2}\left\|\delta_{x}\delta_{y}\widetilde{U}^{m}\right\|^{2}.
\end{aligned}
\end{equation}
Taking \eqref{e4.13}, \eqref{e4.14}, \eqref{e4.15} and \eqref{e4.16} into \eqref{e4.12},
and with the notation
\begin{equation*}
\left\|V\right\|_{\mathcal{B}}^{2}=\left\|V\right\|_{\mathcal{A}_{h}}^{2}+\eta_{0}^{2}\left\|\delta_{x}\delta_{y}V\right\|^{2},
\end{equation*}
we have
\begin{equation*}
\begin{aligned}
c_{0}\left\|\widetilde{U}^{m}\right\|_{\mathcal{B}}^{2}\le& \sum\limits_{n=2}^{m}|\tilde{w}_{1}|\left(\left\|\widetilde{U}^{n}\right\|_{\mathcal{B}}+\left\|\widetilde{U}^{n-1}\right\|_{\mathcal{B}}\right)\left\|\widetilde{U}^{n-1}\right\|_{\mathcal{A}_{h}}\\
&+\sum\limits_{n=3}^{m}\sum\limits_{k=2}^{n-1}|\tilde{w}_{k}-\tilde{w}_{k-1}|\left\|\widetilde{U}^{n-k}\right\|_{\mathcal{B}}\left(\left\|\widetilde{U}^{n-1}\right\|_{\mathcal{A}_{h}}+\left\|\widetilde{U}^{n}\right\|_{\mathcal{A}_{h}}\right)\\
&+2\tau\sum\limits_{n=1}^{m}\left\|\mathcal{F}^{n-\frac{1}{2}}\right\|_{\mathcal{A}_{h}} \left\|\widetilde{U}^{n-\frac{1}{2}} \right\|_{\mathcal{B}}.
\end{aligned}
\end{equation*}
Choosing a suitable $\mathcal{K}$ such that $\left\|\widetilde{U}^{\mathcal{K}}\right\|_{\mathcal{B}}=\max\limits_{1\le n\le N}\left\|\widetilde{U}^{n}\right\|_{\mathcal{B}}$, thus
\begin{equation}\label{e4.17}
\begin{aligned}
c_{0}\left\|\widetilde{U}^{\mathcal{K}}\right\|_{\mathcal{B}}\le& 2\sum\limits_{n=2}^{\mathcal{K}}|\tilde{w}_{1}|\left\|\widetilde{U}^{n-1}\right\|_{\mathcal{A}_{h}}+2\tau\sum\limits_{n=1}^{\mathcal{K}}\left\|\mathcal{F}^{n-\frac{1}{2}}\right\|_{\mathcal{A}_{h}}
+\sum\limits_{n=3}^{\mathcal{K}}\sum\limits_{k=2}^{n-1}|\tilde{w}_{k}-\tilde{w}_{k-1}|\left(\left\|\widetilde{U}^{n-1}\right\|_{\mathcal{A}_{h}}+\left\|\widetilde{U}^{n}\right\|_{\mathcal{A}_{h}}\right)\\
\le& 2\sum\limits_{n=1}^{N-1}|\tilde{w}_{1}|\left\|\widetilde{U}^{n}\right\|_{\mathcal{A}_{h}}+2\tau\sum\limits_{n=1}^{N}\left\|\mathcal{F}^{n-\frac{1}{2}}\right\|_{\mathcal{A}_{h}}
+\sum\limits_{n=3}^{N}\sum\limits_{k=2}^{n-1}|\tilde{w}_{k}-\tilde{w}_{k-1}|\left(\left\|\widetilde{U}^{n-1}\right\|_{\mathcal{A}_{h}}+\left\|\widetilde{U}^{n}\right\|_{\mathcal{A}_{h}}\right).
\end{aligned}
\end{equation}
At this point, we turn our attention to deriving estimates for the coefficients in \eqref{e4.17}. By applying Lemma \ref{lemma3.1} together with the Taylor expansion, we obtain
\begin{equation}\label{e4.18}
\begin{aligned}
\sum\limits_{k=2}^{n-1}|\tilde{w}_{k}-\tilde{w}_{k-1}|&=\frac{1}{2}\sum\limits_{k=2}^{n-1}|w_{k}-w_{k-1}+w_{k-1}-w_{k-2}|\le \sum\limits_{k=1}^{n-1}|w_{k}-w_{k-1}|\\
&\le \tau \sum\limits_{k=1}^{n-1}\int_{t_{k-1}}^{t_{k+1}} |g''(t)|dt \\
&\le Q\tau \sum\limits_{k=1}^{n-1}\int_{t_{k-1}}^{t_{k+1}} (|\ln t|+1)dt
\le Q\tau\int_{0}^{t_n}t^{-\varepsilon}\le Q\tau,
\end{aligned}
\end{equation}
in which, $0< \varepsilon\ll 1$. Meanwhile, in \eqref{e4.18}, we used the fact that
\begin{equation*}
\begin{split}
g(t_{k+1})-2g(t_{k})+g(t_{k-1}) = \int_{t_{k}}^{t_{k+1}} (t_{k+1}-t) g''(t)dt + \int_{t_{k-1}}^{t_{k}} (t-t_{k-1}) g''(t)dt.
\end{split}
\end{equation*}
On the other hand, it is obvious that
\begin{align}
  &|w_0| \leq \int_0^{\tau} |g'(t)|dt \leq Q \tau,\label{e4.19}\\
  &|\tilde{w}_1|= |\frac{w_{1}+w_{0}}{2}|\leq \frac{1}{2}\int_0^{2\tau} |g'(t)|dt \leq Q \tau.\label{e4.20}
\end{align}
Based on \eqref{e4.18}, \eqref{e4.19} and \eqref{e4.20}, with the help
of Lemma \ref{space2}, \eqref{e4.17} turns into
\begin{equation}\label{e4.21}
\begin{split}
\frac{\sqrt{3}}{3}\left\|\widetilde{U}^N\right\| \leq \left\|\widetilde{U}^{\mathcal{K}}\right\|_{\mathcal{A}_{h}}
\leq Q\left( \tau \sum\limits_{n=1}^{N} \left\|\widetilde{U}^n\right\|
+ \tau \sum\limits_{n=1}^{N} \left\| \mathcal{F}^{n-\frac{1}{2}}\right\|\right).
\end{split}
\end{equation}
The use of Gr\"{o}nwall's lemma with
$\beta_{\bar{\alpha}}(t)\in L_{1,\text{loc}}(0,\infty)$ completes the proof.
\end{proof}

Next, we shall deduce the convergence of the fully discrete scheme. Denote
$\tilde{e}_{i,j}^n = \tilde{u}_{i,j}^n - \widetilde{U}_{i,j}^n$ with $(i,j)\in\omega$,
$0\leq n \leq N$. Then, subtracting \eqref{e4.9} from \eqref{eq4.8}, we get the
following error equations
\begin{equation}\label{e4.22}
\begin{aligned}
&c_{0}\delta_t\mathcal{A}_{h}\tilde{e}_{i,j}^{n-\frac{1}{2}}+ \sum\limits_{k=1}^{n-1} \tilde{w}_{n-k} \delta_t\mathcal{A}_{h}\tilde{e}_{i,j}^{k-\frac{1}{2}} - \tau^{\bar{\alpha}} \sum\limits_{p=0}^n \chi_{p}^{(\bar{\alpha})} \Lambda_{h} \tilde{e}_{i,j}^{n-p-\frac{1}{2}} \\
&+\frac{\left(\tau^{\bar{\alpha}+1}\chi_{0}^{(\bar{\alpha})}\right)^{2}}{4c_{0}}\delta_{x}^{2}\delta_{y}^{2}\delta_{t}\tilde{e}_{i,j}^{n-\frac{1}{2}}=\tilde{R}_{i,j}^{n-\frac{1}{2}},\quad (i,j)\in\omega,\quad 1\le n\le N,
\end{aligned}
\end{equation}
\begin{equation}\label{e4.23}
\tilde{e}_{i,j}^{0}=0,\quad (i,j)\in \omega,\qquad  \tilde{e}_{i,j}^{n}=0,\quad (i,j)\in \partial\omega,~0\le n\le N.
\end{equation}
Let $u_{i,j}^n=u(x_{i},y_{j},t_{n})$ be the solution of \eqref{Model2} and $U_{i,j}^n$
be the approximation of $u_{i,j}^n$. Denote $e_{i,j}^n = u_{i,j}^n - U_{i,j}^n$
for $(i,j)\in\omega$, $0\leq n \leq N$. It is clear that $\tilde{e}_{i,j}^n=e_{i,j}^n$.
Based on \eqref{e4.22} and \eqref{e4.23}, we obtain the following convergence results.

\begin{theorem}
Let $\{u_{i,j}^{n}\mid(i,j)\in\omega,~1\le n\le N\}$ be the solution of \eqref{Model2}
and $\{U_{i,j}^{n}\mid(i,j)\in\omega,~1\le n\le N\}$ be its numerical solution.
Then, based on Lemma \ref{lemma3.1} and Assumption \ref{regularity}, we have the error estimate
\begin{equation*}
\max\limits_{1\leq n \leq N} \|u^n-U^n\| \leq Q \left(\tau^{\alpha_0}+h_{1}^{4}+h_{2}^{4}\right).
\end{equation*}
\end{theorem}
\begin{proof}
On the basis of \eqref{e4.21}, in combination with \eqref{e4.22} and \eqref{e4.23},
we have
\begin{equation*}
\begin{split}
\left\|\tilde{e}^N\right\| \leq  Q \left(\tau \sum\limits_{n=1}^{N} \|\tilde{e}^n\|
+ \tau \sum\limits_{n=1}^{N} \left\| \tilde{R}^{n-\frac{1}{2}} \right\|\right),
\end{split}
\end{equation*}
which follows from discrete Gr\"{o}nwall's lemma that
\begin{equation}\label{e4.24}
\left\|\tilde{e}^N\right\| \leq  Q \tau \sum\limits_{n=1}^{N} \left\| (R_1)^{n-\frac{1}{2}} - (R_2)^{n-\frac{1}{2}} + (R_3)^{n-\frac{1}{2}}\right\|+Q\left(\tau^{\alpha_{0}}+h_{1}^{4}+h_{2}^{4}\right).
\end{equation}
Then, we shall analyze all the terms at the right-hand side of \eqref{e4.24}.
At first, based on Assumption \ref{regularity}, \eqref{e4.3} gives
\begin{equation}\label{e4.25}
\begin{split}
\tau \sum\limits_{n=1}^{N} \left\| (R_1)^{n-\frac{1}{2}} \right\| = \tau \left\| (R_1)^{\frac{1}{2}} \right\| +  \tau \sum\limits_{n=2}^{N} \left\| (R_1)^{n-\frac{1}{2}} \right\|
&\leq Q\tau\int_0^{t_1}\left\|\p_{t}^2\tilde{u}(t)\right\|dt+Q\tau^{2}\sum\limits_{n=2}^{N}\int_{t_{n-1}}^{t_n}\left\|\p_{t}^3\tilde{u}(t)\right\|dt\\
& \leq Q\tau\int_0^{\tau} t^{\alpha_0-2}dt + Q \tau^2 \int_{t_1}^{t_N} t^{\alpha_0-3}dt \leq Q\tau^{\alpha_0}.
\end{split}
\end{equation}
Subsequently, in order to bound $(R_2)^{n-\frac{1}{2}}$, we split $(R_2)^n=(R_{21})^n+(R_{22})^n$, where
\begin{align*}
    (R_{21})^n &= \sum\limits_{k=1}^{n}    \int_{t_{k-1}}^{t_k} \Biggl[  \frac{t_{k-1}-s}{\tau} \int_{s}^{t_k} g''(t_n-\theta)(t_k-\theta)d\theta\Biggl]\p_{s}^2 \tilde{u}(s)ds,\\
    (R_{22})^n&=\sum\limits_{k=1}^{n}\int_{t_{k-1}}^{t_k} \Biggl[\frac{s-t_k}{\tau} \int_{s}^{t_{k-1}} g''(t_n-\theta)(t_{k-1}-\theta)d\theta   \Biggl] \p_{s}^2 \tilde{u}(s)ds.
\end{align*}
Then, we further split $(R_{21})^n$ such that
\begin{align*}
    (R_{21})^n =&\sum\limits_{k=1}^{n-1}\int_{t_{k-1}}^{t_k} \Biggl[  \frac{t_{k-1}-s}{\tau} \int_{s}^{t_k} g''(t_n-\theta)(t_k-\theta)d\theta\Biggl]\p_{s}^2 \tilde{u}(s)ds\\
     &+ \int_{t_{n-1}}^{t_n} \Biggl[  \frac{t_{n-1}-s}{\tau} \int_{s}^{t_n} g''(t_n-\theta)(t_n-\theta)d\theta\Biggl]\p_{s}^2 \tilde{u}(s)ds\\
     =&:\mathcal{R}_1^n + \mathcal{R}_2^n.
\end{align*}
Under Assumption \ref{regularity}, Lemma \ref{lemma3.1} helps us yield
\begin{equation*}
\begin{aligned}
|\mathcal{R}_1^n|&\le Q\sum_{k=1}^{n-1}\int_{t_{k-1}}^{t_{k}}\frac{t_{k-1}-s}{\tau}\int_{s}^{t_{k}}(t_{n}-\theta)^{-\varepsilon}(t_{k}-\theta)d\theta \left|\p_{s}^2 \tilde{u}(s)\right|ds\\
&\le Q\tau^{2}\sum_{k=1}^{n-1}\int_{t_{k-1}}^{t_{k}}(t_{n}-s)^{-\varepsilon}s^{\alpha_{0}-2}ds\\
&\leq Q\tau^{2}\int_{0}^{t_{n}}(t_{n}-s)^{-\varepsilon}s^{\alpha_{0}-2}ds
\le Q\tau^{2}t_{n}^{\alpha_{0}-(1+\varepsilon)}\leq Q\tau^{2},
\end{aligned}
\end{equation*}
and
\begin{equation*}
\begin{aligned}
|\mathcal{R}_2^n|&\le Q\int_{t_{n-1}}^{t_{n}}\int_{s}^{t_{k}}(t_{n}-\theta)^{1-\varepsilon}d\theta \left|\p_{s}^2 \tilde{u}(s)\right|ds\\
&\le Q\tau\int_{t_{n-1}}^{t_{n}}(t_{n}-s)^{1-\varepsilon}s^{\alpha_{0}-2}ds\\
&\leq Q\tau^{2-\varepsilon}\int_{t_{n-1}}^{t_{n}}s^{\alpha_{0}-2}ds
\le Q\tau^{2-\varepsilon}\left(t_{n}^{\alpha_{0}-1}-t_{n-1}^{\alpha_{0}-1}\right),
\end{aligned}
\end{equation*}
which indicate that $|(R_{21})^n|\le Q[\tau^2+\tau^{2-\varepsilon}(t_{n}^{\alpha_{0}-1}-t_{n-1}^{\alpha_{0}-1})]$. Applying the same procedure to $(R_{22})^n$ yields an identical estimate. Thus,
\begin{equation}\nonumber
\tau \sum\limits_{n=1}^{N} \left\| (R_2)^n \right\|
\leq Q\left(\tau^{2}+\tau^{3-\varepsilon}T^{\alpha_0-1}\right)\le Q\tau^2,
\end{equation}
which further implies that
\begin{equation}\label{e4.26}
  \tau \sum\limits_{n=1}^{N} \left\| (R_2)^{n-\frac{1}{2}} \right\| \leq Q\tau^{2}.
\end{equation}
In addition, according to Lemma \ref{CQerror}, we have
\begin{equation}\nonumber
\begin{split}
\tau \sum\limits_{n=1}^{N} \left\| (R_3)^n \right\|
\leq & Q\left(\tau \sum\limits_{n=1}^{N} \tau^{2} t_n^{\bar{\alpha}-1}\|\Delta \bar{u}_0\| + \tau^{\bar{\alpha}+2}\int_0^{t_N} s^{\alpha_0-2}ds+ \tau^3 \sum_{n=2}^{N} \int_0^{t_{n-1}}(t_n-s)^{\bar{\alpha}-1} s^{\alpha_0-2}ds\right)\\
\leq & Q\left(\tau^{2}\|\Delta \bar{u}_0\|\int_0^{t_N}s^{\bar{\alpha}-1}ds+\tau^{\alpha_{0}+1}+\tau^{\bar{\alpha}+1}\int_0^{t_N}s^{\alpha_{0}-2}ds\right)\\
\leq & Q\tau^{\alpha_0},
\end{split}
\end{equation}
from which it follows that
\begin{equation}\label{e4.27}
\tau \sum\limits_{n=1}^{N} \left\| (R_3)^{n-\frac{1}{2}} \right\| \leq Q\tau^{\alpha_{0}}.
\end{equation}
Inserting \eqref{e4.25}, \eqref{e4.26} and \eqref{e4.27} into \eqref{e4.24},
we finish the proof of the theorem.
\end{proof}

\section{Accurately second-order ADI compact difference scheme}\label{sec4}

Similarly, assuming that the solution to \eqref{Model2} satisfies the regularity conditions in Assumption \ref{regularity}, we proceed in this section to establish and analyze an ADI compact difference scheme with accurately second-order temporal and fourth-order spatial accuracy.

\subsection{Time semi-discretization}
First, integrating \eqref{Model2} from $t=t_{n-1}$ to $t_{n}$ and multiplying $1/\tau$,
we obtain
\begin{equation}\label{e5.1}
\delta_{t}u^{n-\frac{1}{2}}+\frac{1}{\tau}\int_{t_{n-1}}^{t_{n}}\widehat{G}(t)dt-\frac{1}{\tau}\int_{t_{n-1}}^{t_{n}}\int_{0}^{t}\beta_{\bar{\alpha}}(t-s)\Delta u(s)dsdt=\frac{1}{\tau}\int_{t_{n-1}}^{t_{n}}\bar{f}(t)dt=:\bar{F}^{n},\quad1\leq n\leq N,
\end{equation}
where
\begin{equation*}
\bar{f}(t)= \int_{0}^{t}\beta_{\bar{\alpha}}(t-s)f(s)ds + g(t)\bar{u}_0,\qquad
\widehat{G}(t)=\int_0^tg'(t-s)\partial_s u(s)ds,\qquad t\ge 0.
\end{equation*}
In this case, we have
\begin{equation*}
\begin{aligned}
&\partial_{t}\widehat{G}(t)=g^{\prime}(0)\partial_{t}u(t)+\int_{0}^{t}g^{\prime\prime}(t-s)\partial_{s}u(s)ds,\\
&\partial_{t}^{2}\widehat{G}(t)=g^{\prime}(0)\partial_{t}^{2}u(t)+g^{\prime\prime}(t)\bar{u}_{0}+\int_{0}^{t}g^{\prime\prime}(s)\partial_{s}^{2}u(t-s)ds,
\end{aligned}
\end{equation*}
which, together with Lemma \ref{lemma3.1} and Assumption \ref{regularity}, gives
\begin{equation}\label{e5.2}
    |\partial_{t}^{2}\widehat{G}(t)|\le Q\left(t^{\alpha_0-2}+t^{-\varepsilon}+t^{\alpha_0-1-\varepsilon}\right)\le Qt^{\alpha_0-2},\qquad t\to0^{+}.
\end{equation}
Next, we mainly discuss the second and third terms on the left-hand side of \eqref{e5.1}.
In combination with \eqref{eq4.4}, we utilize the middle rectangle formula to obtain
\begin{equation}\label{e5.3}
\begin{split}
\frac1\tau\int_{t_{n-1}}^{t_n}\widehat{G}(t)dt=\frac{\widehat{G}(t_n)+\widehat{G}(t_{n-1})}{2}+(R_{4})^{n}
=\sum_{k=1}^n\tilde{w}_{n-k}\delta_tu^{k-\frac{1}{2}}+(R_{2}^*)^{n-\frac{1}{2}}+(R_{4})^{n},
\end{split}
\end{equation}
where $(R_{2}^*)^{n}$ is similar to $(R_2)^n$, with the only distinction being the substitution of $\tilde{u}$ by $u$. On the other hand, for the third term, we employ the PI rule to get
\begin{equation}\label{e5.4}
\begin{aligned}
\frac{1}{\tau}\int_{t_{n-1}}^{t_{n}}\int_{0}^{t}\beta_{\bar{\alpha}}(t-s)\Delta u(s)dsdt =\frac{1}{\tau}\int_{t_{n-1}}^{t_{n}}\int_{0}^{t}\beta_{\bar{\alpha}}(t-s)\Delta\breve{u}(s)dsdt+(R_{5})^{n} =\mathcal{P}_{\bar{\alpha}}^{n-\frac{1}{2}}\Delta u+(R_{5})^n,
\end{aligned}
\end{equation}
with
\begin{equation}\label{e5.5}
(R_5)^n=\frac{1}{\tau}\int_{t_{n-1}}^{t_n}\int_0^t\beta_{\bar{\alpha}}(t-s)\left[\Delta u(s)-\Delta\breve{u}(s)\right]dsdt.
\end{equation}
Substituting \eqref{e5.3} and \eqref{e5.4} into \eqref{e5.1}, we have
\begin{equation}\label{e5.6}
\delta_t u^{n-\frac{1}{2}} + \sum\limits_{k=1}^{n} \tilde{w}_{n-k} \delta_t u^{k-\frac{1}{2}}-\mathcal{P}_{\bar{\alpha}}^{n-\frac{1}{2}}\Delta u=\bar{F}^n+(R_*)^n,
\end{equation}
in which $(R_*)^n=(R_5)^n-(R_4)^n-(R_2^*)^{n-\frac{1}{2}}$ for $1\le n\le N$.

Building on the time discretization, we next derive and analyze the corresponding fully discrete scheme.

\subsection{Fully discrete scheme}
Applying the compact operator $\mathcal{A}_{h}$ on both sides of \eqref{e5.6} evaluated at $(x_i,y_j)$, and from Lemma \ref{space3} and Assumption \ref{regularity}, we have
\begin{equation}\label{e5.7}
\delta_t\mathcal{A}_{h}u_{i,j}^{n-\frac{1}{2}} + \sum\limits_{k=1}^{n} \tilde{w}_{n-k} \delta_t \mathcal{A}_{h}u_{i,j}^{k-\frac{1}{2}}-\mathcal{P}_{\bar{\alpha}}^{n-\frac{1}{2}}\Lambda_{h}u_{i,j}=\mathcal{A}_{h}\bar{F}_{i,j}^n+\mathcal{A}_{h}(R_*)_{i,j}^n+\mathcal{O}(h_{1}^{4}+h_{2}^{4}),
\end{equation}
for $(i,j)\in\omega$ and $1\le n\le N$. Similarly, putting the small term
\begin{equation*}
\frac{\tau^{2}\lambda_{1,1}^{2}}{c_{0}}\delta_{x}^{2}\delta_{y}^{2}\delta_{t}u_{i,j}^{\frac{1}{2}}\quad\text{and}\quad \frac{\tau^{2}\lambda_{n,n}^{2}}{4c_{0}}\delta_{x}^{2}\delta_{y}^{2}\delta_{t}u_{i,j}^{n-\frac{1}{2}},\quad 2\le n\le N,
\end{equation*}
into the both sides of \eqref{e5.7}, taking the initial and boundary conditions \eqref{ibc}
into consideration, then, it turns into
\begin{align}
&c_{0}\delta_t\mathcal{A}_{h}u_{i,j}^{\frac{1}{2}} -\lambda_{1,1}\Lambda_{h}u_{i,j}^{1}+\frac{\tau^{2}\lambda_{1,1}^{2}}{c_{0}}\delta_{x}^{2}\delta_{y}^{2}\delta_{t}u_{i,j}^{\frac{1}{2}}=\mathcal{A}_{h}\bar{F}_{i,j}^1+(\tilde{R}_*)_{i,j}^1,\quad (i,j)\in\omega,\label{e5.8}\\
&c_{0}\delta_t\mathcal{A}_{h}u_{i,j}^{n-\frac{1}{2}} + \sum\limits_{k=1}^{n-1} \tilde{w}_{n-k} \delta_t \mathcal{A}_{h}u_{i,j}^{k-\frac{1}{2}}-\mathcal{P}_{\bar{\alpha}}^{n-\frac{1}{2}}\Lambda_{h}u_{i,j}+\frac{\tau^{2}\lambda_{n,n}^{2}}{4c_{0}}\delta_{x}^{2}\delta_{y}^{2}\delta_{t}u_{i,j}^{n-\frac{1}{2}}
=\mathcal{A}_{h}\bar{F}_{i,j}^n+(\tilde{R}_*)_{i,j}^n,\nonumber\\
&\qquad\qquad\qquad\qquad (i,j)\in \omega,\quad 2\le n\le N,\label{e5.9}\\
&u_{i,j}^{0}=u_0(x_i,y_j),\quad (i,j)\in \omega,\qquad  u_{i,j}^{n}=0,\quad (i,j)\in \partial\omega,~0\le n\le N,\label{e5.10}
\end{align}
with $(\tilde{R}_{*})_{i,j}^{n}=\mathcal{A}_h(R_{*})_{i,j}^{n}+\mathcal{O}(\tau^{2}+h_{1}^{4}+h_{2}^{4})$.
Next, omitting the local truncation error term $(\tilde{R}_{*})_{i,j}^{n}$ in
\eqref{e5.8} and \eqref{e5.9}, and replacing $u_{i,j}^{n}$ with its numerical approximation
$U_{i,j}^{n}$, we obtain the accurately second-order ADI compact difference
scheme as follows
\begin{align}
&c_{0}\delta_t\mathcal{A}_{h}U_{i,j}^{\frac{1}{2}} -\lambda_{1,1}\Lambda_{h}U_{i,j}^{1}+\frac{\tau^{2}\lambda_{1,1}^{2}}{c_{0}}\delta_{x}^{2}\delta_{y}^{2}\delta_{t}U_{i,j}^{\frac{1}{2}}=\mathcal{A}_{h}\bar{F}_{i,j}^1,\quad (i,j)\in\omega,\label{e5.11}\\
&c_{0}\delta_t\mathcal{A}_{h}U_{i,j}^{n-\frac{1}{2}} + \sum\limits_{k=1}^{n-1} \tilde{w}_{n-k} \delta_t \mathcal{A}_{h}U_{i,j}^{k-\frac{1}{2}}-\mathcal{P}_{\bar{\alpha}}^{n-\frac{1}{2}}\Lambda_{h}U_{i,j}+\frac{\tau^{2}\lambda_{n,n}^{2}}{4c_{0}}\delta_{x}^{2}\delta_{y}^{2}\delta_{t}U_{i,j}^{n-\frac{1}{2}}=\mathcal{A}_{h}\bar{F}_{i,j}^n,\nonumber\\
&\qquad\qquad\qquad (i,j)\in \omega,\quad 2\le n\le N,\label{e5.12}\\
&U_{i,j}^{0}=u_0(x_i,y_j),\quad (i,j)\in \omega,\qquad  U_{i,j}^{n}=0,\quad (i,j)\in \partial\omega,~0\le n\le N.\label{e5.13}
\end{align}

Ultimately, this yields the familiar ADI form.
\begin{equation}\nonumber
c_{0}\left(\mathcal{A}_{1}-\eta_{n}\delta_{x}^{2}\right)\left(\mathcal{A}_{2}-\eta_{n}\delta_{y}^{2}\right)\mathcal{E}_{i,j}^{n}=\hat{F}_{i,j}^{n},\quad (i,j)\in\omega,\quad 1\le n\le N,
\end{equation}
where
\begin{align*}
\hat{F}_{i,j}^{1}=&\tau\lambda_{1,1}\Lambda_{h}U^{0}_{i,j}+\tau\mathcal{A}_{h}\bar{F}_{i,j}^{1},\\
\hat{F}_{i,j}^{n}=&\tau\lambda_{n,n}\Lambda_{h}U^{n-1}_{i,j}+\tau\lambda_{n,1}\Lambda_{h}U^{1}_{i,j}+\tau\sum_{k=2}^{n-1}\lambda_{n,j}\Lambda_{h}U^{k-\frac{1}{2}}_{i,j}-\sum\limits_{k=1}^{n-1} \tilde{w}_{n-k}\mathcal{A}_{h}\mathcal{E}_{i,j}^{k}+\tau\mathcal{A}_{h}\bar{F}_{i,j}^{n},
\end{align*}
with the notations
\begin{align*}
&\mathcal{E}_{i,j}^{n}=U_{i,j}^n-U_{i,j}^{n-1},\qquad 1\le n\le N,\\
&\eta_{1}=\frac{\tau\lambda_{1,1}}{c_{0}},\quad \eta_{n}=\frac{\tau\lambda_{n,n}}{2c_{0}},\quad n\ge 2.
\end{align*}
As the calculations closely parallel those for the formally second-order ADI compact difference scheme, we omit the details for brevity.

\subsection{Analysis of fully discrete scheme}
Below we shall derive the stability and convergence of the fully discrete ADI compact difference scheme \eqref{e5.11}-\eqref{e5.13}.
\begin{theorem}
Suppose that $\{U_{i,j}^{n}\mid(i,j)\in\omega,~1\le n\le N\}$ is the solution of the accurately second-order ADI compact difference scheme \eqref{e5.11}-\eqref{e5.13}, then it holds that
\begin{equation*}
\left\|U^N\right\| \leq Q \left( \left\|U^{0}\right\|_{\mathcal{C}} + \tau\sum\limits_{n=1}^{N} \|\bar{F}^{n}\| \right) \leq Q \left( \left\|U^{0}\right\|_{\mathcal{C}} + \| \bar{u}_0\| + \max\limits_{0\leq t \leq t_N}\|f(t)\| \right).
\end{equation*}
\end{theorem}
\begin{proof}
Taking the inner product of \eqref{e5.11} and \eqref{e5.12} with $2\tau U^{1}$
and $2\tau U^{n-\frac{1}{2}}$, respectively, and summing up for $n$ from $2$
to $m$ for $2\le m\le N$, we yield
\begin{equation}\label{e5.14}
\begin{aligned}
& 2\tau c_{0}\left(\mathcal{A}_{h}\delta_{t}U^{\frac{1}{2}}, U^{1}\right)+2\tau c_{0}\sum\limits_{n=2}^{m} \left(\mathcal{A}_{h}\delta_{t}U^{n-\frac{1}{2}}, U^{n-\frac{1}{2}} \right)\\
&+2\tau \sum\limits_{n=2}^{m} \sum\limits_{k=1}^{n-1} \tilde{w}_{n-k}  \left(\mathcal{A}_{h}\delta_t U^{k-\frac{1}{2}}, U^{n-\frac{1}{2}} \right) - 2\tau \left(\mathcal{P}_{\bar{\alpha}}^{\frac{1}{2}}\Lambda_{h}U,U^{1}\right)\\
& -2\tau\sum\limits_{n=2}^{m} \left(\mathcal{P}_{\bar{\alpha}}^{n-\frac{1}{2}}\Lambda_{h}U,U^{n-\frac{1}{2}}\right)+\frac{2\tau\left(\tau\lambda_{1,1}\right)^{2}}{c_{0}}\left(\delta_{x}^{2}\delta_{y}^{2}\delta_{t}U^{\frac{1}{2}},U^{1}\right)\\
&+\frac{2\tau\left(\tau\lambda_{n,n}\right)^{2}}{4c_{0}}\sum\limits_{n=2}^{m}\left(\delta_{x}^{2}\delta_{y}^{2}\delta_{t}U^{n-\frac{1}{2}},U^{n-\frac{1}{2}}\right)
=2\tau( \mathcal{A}_{h}\bar{F}^{1},  U^{1} )+2\tau \sum\limits_{n=2}^{m} \left( \mathcal{A}_{h}\bar{F}^{n},  U^{n-\frac{1}{2}} \right).
\end{aligned}
\end{equation}
First, similar to \eqref{e4.13}, it is obvious that
\begin{equation}\label{e5.15}
2\tau c_{0}\left(\mathcal{A}_{h}\delta_{t}U^{\frac{1}{2}}, U^{1}\right)+2\tau c_{0}\sum\limits_{n=2}^{m} \left(\mathcal{A}_{h}\delta_{t}U^{n-\frac{1}{2}}, U^{n-\frac{1}{2}} \right)\ge  c_{0}\left(\left\|U^{m}\right\|_{\mathcal{A}_{h}}^{2}-\left\|U^{0}\right\|_{\mathcal{A}_{h}}^{2}\right).
\end{equation}
Meanwhile, similar to \eqref{e4.14}, we have
\begin{equation}\label{e5.16}
\begin{aligned}
2\tau \sum\limits_{n=2}^{m} \sum\limits_{k=1}^{n-1} \tilde{w}_{n-k}  \left(\mathcal{A}_{h}\delta_t U^{k-\frac{1}{2}}, U^{n-\frac{1}{2}} \right)
\le& 2\sum\limits_{n=2}^{m}|\tilde{w}_{1}|\left\|U^{n-1}\right\|_{\mathcal{A}_{h}}\left\|U^{n-\frac{1}{2}}\right\|_{\mathcal{A}_{h}}
+2\sum\limits_{n=2}^{m}|\tilde{w}_{n-1}|\left\|U^{0}\right\|_{\mathcal{A}_{h}}\left\|U^{n-\frac{1}{2}}\right\|_{\mathcal{A}_{h}}\\
&+2\sum\limits_{n=3}^{m}\sum\limits_{k=2}^{n-1}|\tilde{w}_{k}-\tilde{w}_{k-1}|\left\|U^{n-k}\right\|_{\mathcal{A}_{h}}\left\|U^{n-\frac{1}{2}}\right\|_{\mathcal{A}_{h}}.
\end{aligned}
\end{equation}
Next, Lemma \ref{PIpositive} leads to
\begin{equation}\label{e5.17}
\begin{aligned}
&-\left(\mathcal{P}_{\bar{\alpha}}^{\frac{1}{2}}\Lambda_{h}U,U^{1}\right)-\sum\limits_{n=2}^{m} \left(\mathcal{P}_{\bar{\alpha}}^{n-\frac{1}{2}}\Lambda_{h}U,U^{n-\frac{1}{2}}\right)\\
=&-\left(\mathcal{P}_{\bar{\alpha}}^{\frac{1}{2}}\mathcal{A}_{2}\delta_{x}^{2}U,U^{1}\right)-\sum\limits_{n=2}^{m} \left(\mathcal{P}_{\bar{\alpha}}^{n-\frac{1}{2}}\mathcal{A}_{2}\delta_{x}^{2}U,U^{n-\frac{1}{2}}\right)
 -\left(\mathcal{P}_{\bar{\alpha}}^{\frac{1}{2}}\mathcal{A}_{1}\delta_{y}^{2}U,U^{1}\right)-\sum\limits_{n=2}^{m} \left(\mathcal{P}_{\bar{\alpha}}^{n-\frac{1}{2}}\mathcal{A}_{1}\delta_{y}^{2}U,U^{n-\frac{1}{2}}\right)\\
=& \left(\mathcal{P}_{\bar{\alpha}}^{\frac{1}{2}}\delta_{x}U,\delta_{x}U^{1}\right)_{\mathcal{A}_{2}}+\sum\limits_{n=2}^{m} \left(\mathcal{P}_{\bar{\alpha}}^{n-\frac{1}{2}}\delta_{x}U,\delta_{x}U^{n-\frac{1}{2}}\right)_{\mathcal{A}_{2}}+\left(\mathcal{P}_{\bar{\alpha}}^{\frac{1}{2}}\delta_{y}U,\delta_{y}U^{1}\right)_{\mathcal{A}_{1}}+\sum\limits_{n=2}^{m} \left(\mathcal{P}_{\bar{\alpha}}^{n-\frac{1}{2}}\delta_{y}U,\delta_{y}U^{n-\frac{1}{2}}\right)_{\mathcal{A}_{1}}\\
\geq& 0.
\end{aligned}
\end{equation}
Moreover, by an argument analogous to that used in deriving \eqref{e4.16}, one obtains for the last two terms on the left-hand side of \eqref{e5.14} that
\begin{equation}\label{e5.18}
\begin{aligned}
\frac{2\tau\left(\tau\lambda_{1,1}\right)^{2}}{c_{0}}\left(\delta_{x}^{2}\delta_{y}^{2}\delta_{t}U^{\frac{1}{2}},U^{1}\right)+\frac{2\tau\left(\tau\lambda_{n,n}\right)^{2}}{4c_{0}}\sum\limits_{n=2}^{m}\left(\delta_{x}^{2}\delta_{y}^{2}\delta_{t}U^{n-\frac{1}{2}},U^{n-\frac{1}{2}}\right)
\ge c_{0}\eta_{2}^{2}\left(\left\|\delta_{x}\delta_{y}U^{m}\right\|^{2}-4\left\|\delta_{x}\delta_{y}U^{0}\right\|^{2}\right),
\end{aligned}
\end{equation}
where we have employed the property
\begin{equation*}
\lambda_{1,1}=\lambda_{2,2}=\cdots=\lambda_{N,N}=\frac{\tau^{\bar{\alpha}}}{\Gamma(\bar{\alpha}+2)}.
\end{equation*}
Putting \eqref{e5.15}, \eqref{e5.16}, \eqref{e5.17} and \eqref{e5.18} into \eqref{e5.14},
and with the notation
\begin{equation*}
\left\|V\right\|_{\mathcal{C}}^{2}=\left\|V\right\|_{\mathcal{A}_{h}}^{2}+\eta_{2}^{2}\left\|\delta_{x}\delta_{y}V\right\|^{2},
\end{equation*}
we have
\begin{equation*}
\begin{aligned}
c_{0}\left\|U^{m}\right\|_{\mathcal{C}}^{2}
\le& 2\sum\limits_{n=2}^{m}|\tilde{w}_{1}|\left\|U^{n-\frac{1}{2}}\right\|_{\mathcal{C}}\left\|U^{n-1}\right\|_{\mathcal{A}_h}
+2\sum\limits_{n=2}^{m}|\tilde{w}_{n-1}|\left\|U^{n-\frac{1}{2}}\right\|_{\mathcal{C}}\left\|U^{0}\right\|_{\mathcal{A}_h}\\
&+\sum\limits_{n=3}^{m}\sum\limits_{k=2}^{n-1}|\tilde{w}_{k}-\tilde{w}_{k-1}|\left\|U^{n-k}\right\|_{\mathcal{C}}\left(\left\|U^{n-1}\right\|_{\mathcal{A}_h}+\left\|U^{n}\right\|_{\mathcal{A}_h}\right)\\
&+4c_{0}\left\|U^{0}\right\|_{\mathcal{C}}^{2}
+2\tau\left\|\bar{F}^{1}\right\|_{\mathcal{A}_{h}} \left\|U^{1} \right\|_{\mathcal{C}}+2\tau\sum\limits_{n=2}^{m}\left\|\bar{F}^{n}\right\|_{\mathcal{A}_{h}} \left\|U^{n-\frac{1}{2}} \right\|_{\mathcal{C}}.
\end{aligned}
\end{equation*}
Choosing a suitable $\mathcal{K}$ such that $\left\|U^{\mathcal{K}}\right\|_{\mathcal{C}}=\max\limits_{1\le n\le N}\left\|U^{n}\right\|_{\mathcal{C}}$, thus
\begin{equation}\label{e5.19}
\begin{aligned}
c_{0}\left\|U^{\mathcal{K}}\right\|_{\mathcal{C}}
\le& 2\sum\limits_{n=2}^{\mathcal{K}}|\tilde{w}_{1}|\left\|U^{n-1}\right\|_{\mathcal{A}_h}
+2\sum\limits_{n=2}^{\mathcal{K}}|\tilde{w}_{n-1}|\left\|U^{0}\right\|_{\mathcal{A}_h}
+2\tau\sum\limits_{n=1}^{\mathcal{K}}\left\|\bar{F}^{n}\right\|_{\mathcal{A}_{h}}\\
&+4c_{0}\left\|U^{0}\right\|_{\mathcal{C}}
+\sum\limits_{n=3}^{\mathcal{K}}\sum\limits_{k=2}^{n-1}|\tilde{w}_{k}-\tilde{w}_{k-1}|\left(\left\|U^{n-1}\right\|_{\mathcal{A}_{h}}+\left\|U^{n}\right\|_{\mathcal{A}_{h}}\right)\\
\le& 2\sum\limits_{n=1}^{N-1}|\tilde{w}_{1}|\left\|U^{n}\right\|_{\mathcal{A}_{h}}
+2\sum\limits_{n=1}^{N-1}|\tilde{w}_{n}|\left\|U^{0}\right\|_{\mathcal{A}_h}
+2\tau\sum\limits_{n=1}^{N}\left\|\bar{F}^{n}\right\|_{\mathcal{A}_{h}}\\
&+4c_{0}\left\|U^{0}\right\|_{\mathcal{C}}
+\sum\limits_{n=3}^{N}\sum\limits_{k=2}^{n-1}|\tilde{w}_{k}-\tilde{w}_{k-1}|\left(\left\|U^{n-1}\right\|_{\mathcal{A}_{h}}+\left\|U^{n}\right\|_{\mathcal{A}_{h}}\right).
\end{aligned}
\end{equation}
By \eqref{e4.18}-\eqref{e4.20}, Lemma \ref{space2} and the fact $|\tilde{w}_{n}|=|\frac{w_n+w_{n-1}}{2}|\le\frac{1}{2}\int_{t_{n-1}}^{t_{n+1}}|g'(t)|dt\le Q\tau$, then \eqref{e5.19} turns into
\begin{equation}\label{e5.20}
\frac{\sqrt{3}}{3}\left\|U^N\right\| \leq \left\|U^{\mathcal{K}}\right\|_{\mathcal{A}_{h}}
\leq Q\left( \left\|U^{0}\right\|_{\mathcal{C}} + \tau \sum\limits_{n=1}^{N} \left\|U^n\right\|
+ \tau \sum\limits_{n=1}^{N} \left\| \bar{F}^{n}\right\|\right).
\end{equation}
The use of Gr\"{o}nwall's lemma with
$\beta_{\bar{\alpha}}(t)\in L_{1,\text{loc}}(0,\infty)$ completes the proof.
\end{proof}

Next, we consider the convergence results of this scheme. Subtracting \eqref{e5.11}-\eqref{e5.13} from \eqref{e5.8}-\eqref{e5.10}, we get the following error equations
\begin{align}
&c_{0}\delta_t\mathcal{A}_{h}e_{i,j}^{\frac{1}{2}} -\lambda_{1,1}\Lambda_{h}e_{i,j}^{1}+\frac{\tau^{2}\lambda_{1,1}^{2}}{c_{0}}\delta_{x}^{2}\delta_{y}^{2}\delta_{t}e_{i,j}^{\frac{1}{2}}=(\tilde{R}_*)_{i,j}^1,\label{e5.21}\\
&c_{0}\delta_t\mathcal{A}_{h}e_{i,j}^{n-\frac{1}{2}} + \sum\limits_{k=1}^{n-1} \tilde{w}_{n-k} \delta_t \mathcal{A}_{h}e_{i,j}^{k-\frac{1}{2}}-\mathcal{P}_{\bar{\alpha}}^{n-\frac{1}{2}}\Lambda_{h}e_{i,j}+\frac{\tau^{2}\lambda_{n,n}^{2}}{4c_{0}}\delta_{x}^{2}\delta_{y}^{2}\delta_{t}e_{i,j}^{n-\frac{1}{2}}=(\tilde{R}_*)_{i,j}^n,\label{e5.22}\\
&e_{i,j}^{0}=0,\quad (i,j)\in \omega,\qquad  e_{i,j}^{n}=0,\quad (i,j)\in \partial\omega,~0\le n\le N,\label{e5.23}
\end{align}
with $e_{i,j}^n = u_{i,j}^n - U_{i,j}^n$ for $(i,j)\in\omega$,
$0\leq n \leq N$.

\vskip 1mm
\begin{theorem}
Let $\{u_{i,j}^{n}\mid(i,j)\in\omega,~1\le n\le N\}$ and
$\{U_{i,j}^{n}\mid(i,j)\in\omega,~1\le n\le N\}$
be the solutions of \eqref{Model2} and \eqref{e5.11}-\eqref{e5.13}, respectively.
Then, based on Lemmas \ref{lemma3.1} and Assumption \ref{regularity}, we have the error estimate
\begin{equation*}
\max\limits_{1\leq n \leq N} \|u^n-U^n\| \leq Q \left(\tau^{2}+h_{1}^{4}+h_{2}^{4}\right).
\end{equation*}
\end{theorem}
\begin{proof}
On the basis of \eqref{e5.20}, in combination with \eqref{e5.21}, \eqref{e5.22} and
\eqref{e5.23}, we have
\begin{equation*}
\left\|e^N\right\| \leq  Q \left(\tau \sum\limits_{n=1}^{N} \|e^n\|
+ \tau \sum\limits_{n=1}^{N} \left\| (\tilde{R}_{*})^{n} \right\|\right),
\end{equation*}
which follows from discrete Gr\"{o}nwall's lemma that
\begin{equation}\label{e5.24}
\left\|e^N\right\| \leq  Q \tau \sum\limits_{n=1}^{N} \left\| (R_5)^{n} - (R_4)^{n} - (R_2^*)^{n-\frac{1}{2}}\right\|+Q\left(\tau^{2}+h_{1}^{4}+h_{2}^{4}\right).
\end{equation}
Then, we shall analyze all the terms on the right-hand side of \eqref{e5.24}. From
\eqref{e5.3}, it is obvious that
\begin{equation*}
(R_4)^n=\frac{1}{\tau}\int_{t_{n-1}}^{t_n}\widehat{G}(t)dt-\widehat{G}(t_{n-\frac{1}{2}})+\left[\widehat{G}(t_{n-\frac{1}{2}})-\frac{\widehat{G}(t_n)+\widehat{G}(t_{n-1})}{2}\right].
\end{equation*}
Hence, by considering \eqref{e5.2} together with the Taylor expansion, we arrive at
\begin{equation}\label{e5.25}
\begin{aligned}
\tau\sum\limits_{n=1}^N\|(R_4)^n\|&\leq\tau\int_0^{t_{\frac{1}{2}}}t\left\|\widehat{G}''(t)\right\|dt+\tau^2\int_{t_{\frac{1}{2}}}^{t_1}\left\|\widehat{G}''(t)\right\|dt+2\sum\limits_{n=2}^N\tau^2\int_{t_{n-1}}^{t_n}\left\|\widehat{G}''(t)\right\|dt\\
&\leq Q\tau\int_{0}^{t_{\frac{1}{2}}}t^{\alpha_{0}-1}dt+Q\tau^{2}\int_{t_{\frac{1}{2}}}^{t_{1}}t^{\alpha_{0}-2}dt+Q\tau^{2}\sum_{n=2}^{N}\int_{t_{n-1}}^{t_{n}}t^{\alpha_{0}-2}dt\leq Q\tau^{2}.
\end{aligned}
\end{equation}
Next, we focus on the estimation of $(R_{5})^{n}$. Splitting \eqref{e5.5} as follows
\begin{equation*}
(R_{5})^{n}=\frac{1}{\tau}\int_{t_{n-1}}^{t_{n}}\int_{0}^{t}\beta_{\bar{\alpha}}(t-s)\left[\Delta u(s)-\Delta u^{*}(s)\right]dsdt+\frac{1}{\tau}\int_{t_{n-1}}^{t_{n}}\int_{0}^{t}\beta_{\bar{\alpha}}(t-s)\left[\Delta u^{*}(s)-\Delta\breve{u}(s)\right]dsdt,
\end{equation*}
with
\begin{equation*}
u^*(s)=
\begin{cases}
u(t_1),&0<s<\tau,\\
\tau^{-1}[(t_i-s)u(t_{i-1})+(s-t_{i-1})u(t_i)],&t_{i-1}<s<t_i,\quad i\ge2.
\end{cases}
\end{equation*}
Subsequently, according to \cite{Mclean} and Assumption \ref{regularity}, we yield
\begin{equation}\label{e5.26}
\begin{aligned}
\tau\sum_{n=1}^N\|(R_5)^n\|\le& Q_{\bar{\alpha},T}\left(\int_{0}^{t_{1}}t\left\|\Delta\partial_{t}u(t)\right\|dt+\tau^{2}\int_{t_{1}}^{t_{N}}\left\|\Delta\partial_{t}^{2}u(t)\right\|dt\right)\\
&+Q_{\bar{\alpha},T}\left(\tau\int_{t_{1}}^{t_{2}}\|\Delta\partial_{t}u(t)\|\:dt+\tau^{2}\int_{t_{1}}^{t_{N}}\left\|\Delta\partial_{t}^{2}u(t)\right\|dt\right)\\
\leq &Q\tau\int_0^{2\tau}t^{\alpha_0-1}dt+Q\tau^2\int_\tau^{t_N}t^{\alpha_0-2}dt
\leq Q\tau^2.
\end{aligned}
\end{equation}
Finally, substituting \eqref{e5.25},\eqref{e5.26} and \eqref{e4.26} into \eqref{e5.24},
it turns into
\begin{equation}\nonumber
\left\|e^N\right\| \leq Q\left(\tau^{2}+h_{1}^{4}+h_{2}^{4}\right).
\end{equation}
Thus, the proof is completed.
\end{proof}

\section{Numerical experiment}\label{sec5}
In this section, we carry out numerical experiments to validate the accuracy and the efficiency of two proposed fully discrete schemes, i.e., formally second-order ADI compact difference scheme (F2OACD) and accurately second-order ADI compact difference scheme (A2OACD). Below we take $\Omega=(0,1)\times (0,1)$, $h=h_1=h_2=\frac{1}{M}$ and $T=1$, with $M=M_{1}=M_{2}$.

To measure the convergence of the proposed methods, we denote the errors and convergence rates in time
\begin{equation*}
E(\tau,h)=\max\limits_{0\le n\le N}\sqrt{h^{2}\sum_{i=1}^{M-1}\sum_{j=1}^{M-1}\left|U_{i,j}^{n}(\tau,h)-U_{i,j}^{2n}(\frac{\tau}{2},h)\right|^{2}},~Rate^{\tau}=\log_2\left(\frac{E(2\tau,h)}{E(\tau,h)}\right),
\end{equation*}
along with the errors and convergence rates in space
\begin{equation*}
S(\tau,h)=\sqrt{h^{2}\sum_{i=1}^{M-1}\sum_{j=1}^{M-1}\left|U_{i,j}^{N}(\tau,h)-U_{2i,2j}^{N}(\tau,\frac{h}{2})\right|^{2}},~Rate^{h}=\log_2\left(\frac{S(\tau,2h)}{S(\tau,h)}\right).
\end{equation*}

\vskip 1mm
\textbf{Example 1.}
In this example, we test the case with $u_{0}(x,y)=\sin(\pi x)\sin(\pi y)$,
$\bar{u}_{0}(x,y)=\sin(2\pi x)\sin(2\pi y)$ and $f(x,y,t)=0$. Meanwhile, to
satisfy the very weak conditions $\alpha'(0)= 0$, we choose
$\alpha(t)=\alpha_{0}+\frac{1}{11}t^{2}$ and limit $1<\alpha_{0}\le 1.9$ such that
$\alpha(t)\in (1,2)$.

\begin{table}\footnotesize\centering
\caption{Discrete errors and temporal convergence rates by fixing $M=16$ with different $\alpha_{0}$ for Example 1.}
\label{table1}  
\begin{tabular}{ccccccc}
\hline\noalign{\smallskip}
\multirow{2}{*}{$\alpha_{0}$} &\multirow{2}{*}{$N$} &\multicolumn{2}{c}{F2OACD scheme} & &  \multicolumn{2}{c}{A2OACD scheme} \\
\cline{3-4} \cline{6-7}
& &   $E(\tau,h)$    & $Rate^{\tau}$    & & $E(\tau,h)$    & $Rate^{\tau}$ \\
\noalign{\smallskip}\hline\noalign{\smallskip}
       & 16  & 5.9599e-02 & *    & & 4.8255e-02 & *    \\
       & 32  & 2.5695e-02 & 1.21 & & 1.3225e-02 & 1.87 \\
$1.20$ & 64  & 1.0927e-02 & 1.23 & & 3.5678e-03 & 1.89 \\
       & 128 & 4.7918e-03 & 1.19 & & 8.6806e-04 & 2.04 \\
       & 256 & 2.1244e-03 & 1.17 & & 2.0957e-04 & 2.05 \\
\noalign{\smallskip}\hline\noalign{\smallskip}
       & 32  & 4.0199e-03 & *    & & 4.0991e-03 & *	\\
       & 64  & 1.2671e-03 & 1.67 & & 8.7841e-04 & 2.22 \\
$1.50$ & 128 & 4.1822e-04 & 1.60 & & 2.0193e-04 & 2.12 \\
       & 256 & 1.4427e-04 & 1.54 & & 4.8330e-05 & 2.06 \\
       & 512 & 5.1424e-05 & 1.49 & & 1.1830e-05 & 2.03 \\
\noalign{\smallskip}\hline\noalign{\smallskip}
       & 8   & 7.0889e-02 & *    & & 5.6004e-02 & *    \\
       & 16  & 1.8802e-02 & 1.91 & & 1.3043e-02 & 2.10 \\
$1.90$ & 32  & 4.8048e-03 & 1.96 & & 3.0786e-03 & 2.08 \\
       & 64  & 1.2164e-03 & 1.98 & & 7.3910e-04 & 2.06 \\
       & 128 & 3.0618e-04 & 1.99 & & 1.7691e-04 & 2.06 \\
\noalign{\smallskip}\hline
\end{tabular}
\end{table}

\begin{table}\footnotesize\centering
\caption{Discrete errors and spatial convergence rates by fixing $N=32$ with different $\alpha_{0}$ for Example 1.}
\label{table2}  
\begin{tabular}{ccccccc}
\hline\noalign{\smallskip}
\multirow{2}{*}{$\alpha_{0}$} &\multirow{2}{*}{$M$} &\multicolumn{2}{c}{F2OACD scheme} & &  \multicolumn{2}{c}{A2OACD scheme} \\
\cline{3-4} \cline{6-7}
& &   $S(\tau,h)$    & $Rate^{h}$    & & $S(\tau,h)$    & $Rate^{h}$ \\
\noalign{\smallskip}\hline\noalign{\smallskip}
       & 4  & 8.1301e-04 & *    & & 1.1750e-04 & *    \\
       & 8  & 4.9771e-05 & 4.03 & & 6.7722e-06 & 4.12 \\
$1.30$ & 16 & 3.0949e-06 & 4.01 & & 4.1509e-07 & 4.03 \\
       & 32 & 1.9318e-07 & 4.00 & & 2.5818e-08 & 4.01 \\
       & 64 & 1.2070e-08 & 4.00 & & 1.6117e-09 & 4.00 \\
\noalign{\smallskip}\hline\noalign{\smallskip}
       & 4  & 7.1109e-04 & *    & & 4.7537e-05 & *	  \\
       & 8  & 4.3471e-05 & 4.03 & & 2.9915e-06 & 3.99 \\
$1.60$ & 16 & 2.7032e-06 & 4.01 & & 1.8425e-07 & 4.02 \\
       & 32 & 1.6874e-07 & 4.00 & & 1.1464e-08 & 4.01 \\
       & 64 & 1.0543e-08 & 4.00 & & 7.1564e-10 & 4.00 \\
\noalign{\smallskip}\hline\noalign{\smallskip}
       & 4  & 4.5057e-03 & *    & & 3.0726e-03 & *    \\
       & 8  & 2.6665e-04 & 4.08 & & 1.7529e-04 & 4.13 \\
$1.90$ & 16 & 1.6485e-05 & 4.02 & & 1.0763e-05 & 4.03 \\
       & 32 & 1.0277e-06 & 4.00 & & 6.6997e-07 & 4.01 \\
       & 64 & 6.4188e-08 & 4.00 & & 4.1831e-08 & 4.00 \\
\noalign{\smallskip}\hline
\end{tabular}
\end{table}

\vskip 1mm
In Table \ref{table1}, we present the errors and temporal convergence rates
when simulating the problem \eqref{VtFDEs}-\eqref{ibc}, for both the F2OACD scheme
and the A2OACD scheme, where we use a very suitable spatial step size by fixing $M=16$
at $\alpha_{0}\in\{1.20,1.50,1.90\}$. Then, we observe from these results that the F2OACD
scheme has the $\alpha_{0}$-order temporal accuracy and the A2OACD scheme can
exhibit a rate of 2 in the temporal direction, which substantiates the
theoretical analysis.

Subsequently, we investigate the performance in space of the F2OACD
scheme and the A2OACD scheme in Table \ref{table2}. This table displays errors and
spatial convergence rates of both schemes for fixed $N=32$ when $\alpha_{0}=1.3, 1.6, 1.9$,
respectively. As can be seen from the numerical results, both schemes have
fourth-order spatial convergence rates. Both of these findings align with
the theoretical predictions.

In \cite{Qiu1}, for the one-dimensional diffusion-wave model related to
\eqref{VtFDEs}-\eqref{ibc}, the authors used the second-order BDF formula
combined with the second-order convolution quadrature rule for the temporal
discretization, to construct a numerical scheme of order $\alpha_{0}$.
It is not difficult to observe that the F2OACD scheme in our paper is similar to
this scheme in terms of time discretization, except that we utilize the
Crank-Nicolson method combined with the trapezoidal convolution quadrature rule.

In the following, we shall present some numerical results of the $\alpha_0$-order
scheme in \cite{Qiu1} and compare them with our proposed F2OACD scheme. To this end, the compact difference method is applied to discretize the spatial direction, thereby constructing a fully discrete scheme for comparison. Meanwhile, we apply the ADI technique to accelerate the computation.

\begin{table}\footnotesize\centering
\caption{Comparison between discrete errors and temporal convergence rates of the
$\alpha_0$-order scheme in \cite{Qiu1} and the F2OACD scheme with $M=16$ for
Example 1.}
\label{table3}  
\resizebox{\textwidth}{!}{
\begin{tabular}{cccccccccccc}
\hline\noalign{\smallskip}
\multirow{3}{*}{$N$} &\multicolumn{5}{c}{$\alpha_0=1.50$} & &\multicolumn{5}{c}{$\alpha_0=1.75$} \\
\cline{2-6} \cline{8-12}
& \multicolumn{2}{c}{$\alpha_0$-order scheme in \cite{Qiu1}} & & \multicolumn{2}{c}{F2OACD scheme}& & \multicolumn{2}{c}{$\alpha_0$-order scheme in \cite{Qiu1}} & & \multicolumn{2}{c}{F2OACD scheme} \\
\cline{2-3} \cline{5-6} \cline{8-9} \cline{11-12}
& $E(\tau,h)$    & $Rate^{\tau}$    & & $E(\tau,h)$    & $Rate^{\tau}$ && $E(\tau,h)$    & $Rate^{\tau}$    & & $E(\tau,h)$    & $Rate^{\tau}$ \\
\noalign{\smallskip}\hline\noalign{\smallskip}
128  & 3.7104e-03 & *    & & 4.1822e-04 & *    & & 1.6912e-03 & *    & & 2.1335e-04 & *    \\
256  & 1.3848e-03 & 1.42 & & 1.4427e-04 & 1.54 & & 5.1165e-04 & 1.72 & & 6.3718e-05 & 1.74 \\
512  & 4.9975e-04 & 1.47 & & 5.1424e-05 & 1.49 & & 1.5291e-04 & 1.74 & & 1.8969e-05 & 1.75 \\
1024 & 1.7797e-04 & 1.49 & & 1.8644e-05 & 1.46 & & 4.5547e-05 & 1.75 & & 5.6418e-06 & 1.75 \\
2048 & 6.3084e-05 & 1.50 & & 6.7951e-06 & 1.46 & & 1.3551e-05 & 1.75 & & 1.6775e-06 & 1.75 \\
\noalign{\smallskip}\hline
\end{tabular}}
\end{table}

Table \ref{table3} lists and compares the discrete errors and associated
temporal convergence rates of the $\alpha_0$-order scheme in \cite{Qiu1}
and the F2OACD scheme by taking different values of $\alpha_0$ when fixed $M=16$.
As can be seen from the results in this table, under the same temporal mesh step size, the discrete errors of F2OACD scheme is much smaller than that of $\alpha_0$-order scheme in \cite{Qiu1}. More precisely, when simulating the model \eqref{VtFDEs}-\eqref{ibc}, the discrete errors produced by these two numerical
schemes differ by nearly an order of magnitude. Thus, this is one of our significant advantages over the existing method in \cite{Qiu1}.

\vskip 2mm
\textbf{Example 2.}
In this example, let $u_{0}(x,y)=\sin(\pi x)\sin(\pi y)$,
$\bar{u}_{0}(x,y)=0$ and $f(x,y,t)=0$. Then, we select $\alpha(t)=\alpha_{0}+\frac{t^2}{3+e^{2t}}$ such that $\alpha'(0)=0$. Besides, for satisfying $\alpha(t)\in (1,2)$,
we limit $1<\alpha_{0}\le 1.9$.

Here, in order to highlight the advantages of proposed ADI schemes for solving problem
\eqref{VtFDEs}-\eqref{ibc} in two-dimensional case, we shall provide corresponding numerical results of formally second-order standard compact difference scheme and accurately second-order standard compact difference scheme. For the sake of clarity, the two methods are presented directly below.

Here, to highlight the advantages of the proposed ADI schemes for solving problem \eqref{VtFDEs}–\eqref{ibc} in two dimensions, we compare them with the formally second-order and the accurately second-order standard compact difference schemes. For clarity, we briefly outline these two reference methods below.

\textbf{I. Formally second-order standard compact difference (F2OSCD):}
\begin{align*}
&\delta_t\mathcal{A}_{h}\widetilde{U}_{i,j}^{n-\frac{1}{2}} + \sum\limits_{k=1}^{n} \tilde{w}_{n-k} \delta_t\mathcal{A}_{h}\widetilde{U}_{i,j}^{k-\frac{1}{2}} - \tau^{\bar{\alpha}} \left( \sum\limits_{p=0}^n \chi_{p}^{(\bar{\alpha})} \Lambda_{h} \widetilde{U}_{i,j}^{n-p-\frac{1}{2}} +  \hat{\rho}_n^{(\bar{\alpha})} \Lambda_{h} \widetilde{U}_{i,j}^0 \right)\\
=&\mathcal{A}_{h}\mathcal{F}_{i,j}^{n-\frac{1}{2}},\qquad (i,j)\in\omega,\quad 1\le n\le N,\\
&\widetilde{U}_{i,j}^{0}=0,\quad (i,j)\in \omega,\qquad  \widetilde{U}_{i,j}^{n}=0,\quad (i,j)\in \partial\omega,~0\le n\le N,\\
&U_{i,j}^n=\widetilde{U}_{i,j}^{n}+u_0(x_i,y_j),\qquad (i,j)\in\omega,\quad 1\le n\le N.
\end{align*}

\textbf{II. Accurately second-order standard compact difference (A2OSCD):}
\begin{align*}
&\delta_t\mathcal{A}_{h}U_{i,j}^{n-\frac{1}{2}} + \sum\limits_{k=1}^{n} \tilde{w}_{n-k} \delta_t \mathcal{A}_{h}U_{i,j}^{k-\frac{1}{2}}-\mathcal{P}_{\bar{\alpha}}^{n-\frac{1}{2}}\Lambda_{h}U_{i,j}=\mathcal{A}_{h}\bar{F}_{i,j}^n,\quad (i,j)\in\omega,\; 1\le n\le N,\\
&U_{i,j}^{0}=u_0(x_i,y_j),\quad (i,j)\in \omega,\qquad  U_{i,j}^{n}=0,\quad (i,j)\in \partial\omega,\quad 0\le n\le N.
\end{align*}
At this stage, we apply the four numerical schemes to problem \eqref{VtFDEs}–\eqref{ibc}, respectively.

\begin{figure}
\centering
\includegraphics[width=.6\textwidth]{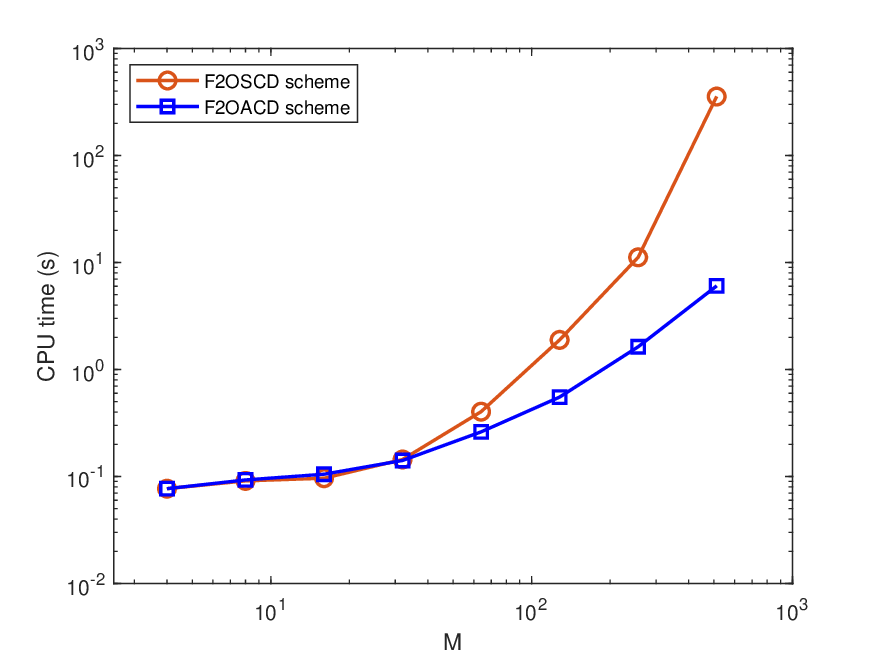}
\caption{The CPU time of F2OSCD and F2OACD schemes under $N=32$, $\alpha_{0}=1.3$ for Example 2.}
\label{Fig1}
\end{figure}

\begin{figure}
\centering
\includegraphics[width=.6\textwidth]{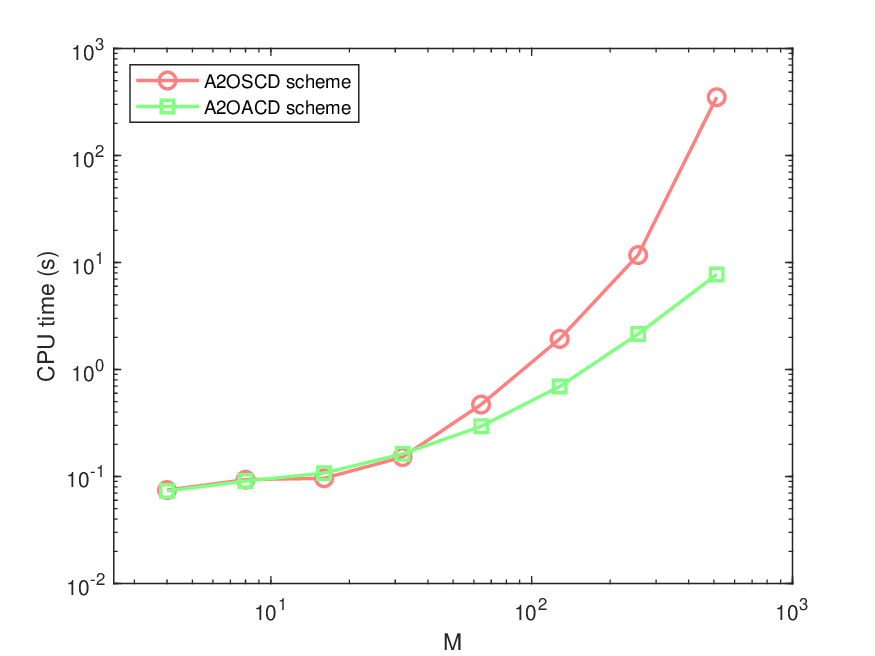}
\caption{The CPU time of A2OSCD and A2OACD schemes under $N=32$, $\alpha_{0}=1.3$ for Example 2.}
\label{Fig2}
\end{figure}

To illustrate the acceleration effect of the ADI methods, we plot the comparisons of CPU time
between the F2OSCD scheme and F2OACD scheme for different values of $M$ under the condition of $N=32$, $\alpha_{0}=1.3$ in Figure \ref{Fig1}. Meanwhile, under the same parameter selection, the CPU time comparisons between the A2OSCD scheme and A2OACD scheme are shown in Figure \ref{Fig2}. It can be clearly seen from Figures \ref{Fig1} and \ref{Fig2} that when $M$ is large, the ADI methods can greatly save the computational cost.

\section*{Declarations}

\noindent {\textbf{Conflict of interest}}: The authors declare no competing financial interests or personal relationships that could have influenced this work.

\vskip 2mm
\noindent {\textbf{Acknowledgements}}: The authors would like to express their sincere gratitude to Prof. Xiangcheng Zheng for his valuable suggestions and insightful discussions on potential improvement strategies for this paper.

\vskip 2mm
\noindent {\textbf{Funding}}: This work is supported by the Scientific Research Fund Project of Yunnan Provincial Education Department (No. 2024J0642), the Yunnan Fundamental Research Projects (No. 202401AU070104), the Scientific Research Fund Project of Yunnan University of Finance and Economics (No. 2024D38), and Postdoctoral Fellowship Program of CPSF (No. GZC20240938).

\vskip 2mm
\noindent {\textbf{Data Availability}}: The datasets are available from the corresponding author upon reasonable request.


\end{document}